\documentclass[12pt]{amsart}
\usepackage{txfonts}      
\usepackage{amssymb}
\usepackage{eucal}
\usepackage{amsmath}
\usepackage{amscd}
\usepackage{xcolor}
\usepackage{multicol}
\usepackage[all]{xy}           
\usepackage{graphicx}
\usepackage{color}
\usepackage{colordvi}
\usepackage{xspace}
\usepackage{tikz}
\usepackage{makecell}
\usepackage{appendix}
\usepackage{enumerate,enumitem}
\usepackage{amsthm}
\usepackage[thicklines]{cancel}
\usepackage[compress]{cite}
\usepackage{ifpdf}
\ifpdf
\usepackage[colorlinks,final,backref=page,hyperindex]{hyperref}
\else
\usepackage[colorlinks,final,backref=page,hyperindex,hypertex]{hyperref}
\fi

\usepackage[active]{srcltx} 
\newtheorem{thm}{Theorem}[section]
\newtheorem{lem}[thm]{Lemma}

\newtheorem{pro}[thm]{Proposition}
\theoremstyle{definition}
\newtheorem{defi}[thm]{Definition}
\newtheorem{ex}[thm]{Example}
\newtheorem{rmk}[thm]{Remark}

\title[Infinite-dimensional pre-Lie bialgebras constructions] {
Infinite-dimensional pre-Lie bialgebras induced from Leibniz-dendriform bialgebras and Zinbiel-dendriform bialgebras
}

\author{Qinxiu Sun}
\address{Qinxiu Sun, Department of Mathematics, Zhejiang University of Science and Technology, Hangzhou, 310023} \email{qxsun@126.com}


\begin{document}
\begin{abstract}  
In this paper, we establish a completed pre-Lie bialgebra structure on the tensor product of
 a Leibniz-dendriform bialgebra and a quadratic $\mathbb{Z}$-graded Zinbiel algebra. 
 We also obtain such a structure on the tensor product of a Zinbiel-dendriform bialgebra and a quadratic 
$\mathbb{Z}$-graded Leibniz algebra. Moreover, a Zinbiel-dendriform bialgebra is precisely one 
whose affinization by a special quadratic $\mathbb{Z}$-graded Leibniz algebra is a completed pre-Lie bialgebra.
 Finally, using solutions of the ZD-YBE (resp.~LD-YBE) with invariant skew-symmetric
  parts in a Zinbiel-dendriform (resp.~ Leibniz-dendriform) algebra, 
  we construct completed solutions possessing invariant symmetric parts of the 
$S$-equation in the induced pre-Lie algebra.

\end{abstract}

\maketitle

\vspace{-1.2cm}

\tableofcontents

\vspace{-1.2cm}

\allowdisplaybreaks

\section{Introduction}

The purpose of this paper is to construct completed pre-Lie bialgebras 
from Leibniz-dendriform bialgebras (resp.~ Zinbiel-dendriform bialgebras) and quadratic $\mathbb{Z}$-graded Zinbiel (resp.~ Leibniz) algebras.

\subsection{Operadic Koszul duality and affinization}
A classical result in operadic theory asserts that the tensor product of algebras over 
any Koszul dual pair of binary quadratic operads carries a natural Lie algebra structure~\cite{GK,LV}. 
This observation provides a powerful and general framework for constructing Lie algebras from pairs 
of Koszul dual algebraic structures. Building upon this operadic principle, Hong, Bai 
and Guo introduced the affinization of Novikov bialgebras, 
showing that the tensor product of a Novikov bialgebra and a quadratic 
?$\mathbb{Z}$-graded right Novikov algebra admits a natural completed Lie bialgebra structure~\cite{HBG}. 
Since every pre-Novikov algebra gives rise to a Novikov algebra and the commutator
 of a pre-Lie algebra is a Lie algebra, it follows from \cite{HBG} that 
 there should exist a natural construction of pre-Lie algebras from pre-Novikov algebras and right Novikov algebras.
  Completed pre-Lie bialgebras are constructed via the affinization of pre-Novikov bialgebras in \cite{LH}. 
 Note that the operads of perm algebras and pre-Lie algebras are Koszul dual to each other \cite{CL}. 
 Subsequently, Lin, Zhou and Bai extended this construction to perm bialgebras and pre-Lie bialgebras, 
 establishing that the tensor product of a perm bialgebra and a quadratic pre-Lie algebra 
 (resp.~a pre-Lie bialgebra and a quadratic perm algebra) carries a Lie bialgebra structure~\cite{LZB}.
  Likewise, since the operads of Leibniz algebras and Zinbiel algebras are 
  Koszul dual \cite{L2}, the tensor product of a Leibniz algebra and
   a Zinbiel algebra also acquires a Lie algebra structure. 
   These constructions and their affinization were further generalized to the bialgebra setting in \cite{HL}. 
 These developments highlight the ubiquity of the tensor product construction in the theory 
 of bialgebras and emphasize the role of Koszul duality in generating new algebraic structures from existing ones.

\subsection{Zinbiel-dendriform bialgebras (commutative quadri-bialgebras) and Leibniz-dendriform bialgebras}   
 A Rota-Baxter operator on an associative algebra, which originated in probability theory \cite{Bax}
 and later found use in combinatorics, 
 naturally endows the underlying vector space with a dendriform algebra structure \cite{Ag1, Ag2, E1}.
  This construction has been extended to other settings, such as Rota-Baxter operators on Leibniz algebras
   and on dendriform algebras give rise to Leibniz-dendriform algebras and quadri-algebras, respectively. 
   Inspired by Lie bialgebras\cite{CP,D}, bialgebra structures for quadri-algebras were studied in [NB], 
   where finite-dimensional quadri-bialgebras
  were shown to be equivalent to Manin triples of dendriform algebras with a 2-cocycle and to Manin triples
  of quadri-algebras with an invariant bilinear form. On the other hand, Leibniz-dendriform algebras
(whose two binary operations sum to a Leibniz algebra \cite{TS}) have been investigated by Sun and Guo \cite{SG},
 they developed the corresponding bialgebra theory, established the equivalence with Manin triples, 
 and introduced the Leibniz-dendriform Yang-Baxter equation (LD-YBE) as an analogue of the classical Yang-Baxter equation. 

\subsection{ Main results of this paper} 

Building on the aforementioned works, we construct pre-Lie (co)algebras from Leibniz-dendriform (co)algebras 
paired with Zinbiel (co)algebras, and dually from Zinbiel-dendriform (co)algebras 
paired with Leibniz (co)algebras, see Propositions \ref{Tb1} and \ref{Tb2} and Propositions \ref{Tb5}-\ref{Tb6}. 
We then lift these constructions to the bialgebra setting. Since finite-dimensional pre-Lie bialgebras 
   are equivalent to para-K$\ddot{\rm{a}}$hler Lie algebras, which yield quadratic pre-Lie algebras \cite{B}. 
   To obtain a pre-Lie bialgebra, we begin by seeking a quadratic pre-Lie algebras on the tensor product of
   a quadratic Leibniz-dendriform and a Zinbiel algebra (or dually, a quadratic Zinbiel-dendriform algebra and 
  Leibniz algebra), equipped with a special nondegenerate bilinear form. Indeed, we verify that
  the tensor product of
   a quadratic Leibniz-dendriform (resp.~Zinbiel-dendriform) algebra and a quadratic Zinbiel (resp.~Leibniz) algebra
  is a quadratic pre-Lie algebra, see Propositions \ref{Sp2} and \ref{Sp3}. 
  Thus, we construct completed pre-Lie bialgebras via
  Leibniz-dendriform (resp.~Zinbiel-dendriform) bialgebras and
   quadratic $\mathbb{Z}$-graded Zinbiel (resp.~Leibniz) algebras (Theorem~\ref{Tb3} and Theorem~\ref{Tb7}).
  In the special case where the $\mathbb{Z}$-graded Leibniz algebra is taken to be a particular one (see Example~\ref{Lq1}), 
this construction yields an affinization of Zinbiel-dendriform bialgebras (Theorem~\ref{Tb7}).
Subsequently, we show that in the induced pre-Lie algebra, solutions of the ZD-YBE carrying invariant skew-symmetric parts
give rise to solutions of the solutions of the 
$S$-equation possessing invariant symmetric parts (Theorem~\ref{Tb4} and Theorem~\ref{Tb8}).
Starting from a solution of the LD-YBE (~resp.~ZD-YBE) with invariant skew-symmetric parts, 
we obtain two completed pre-Lie bialgebras via two distinct methods, and establish their relations as follows:
 
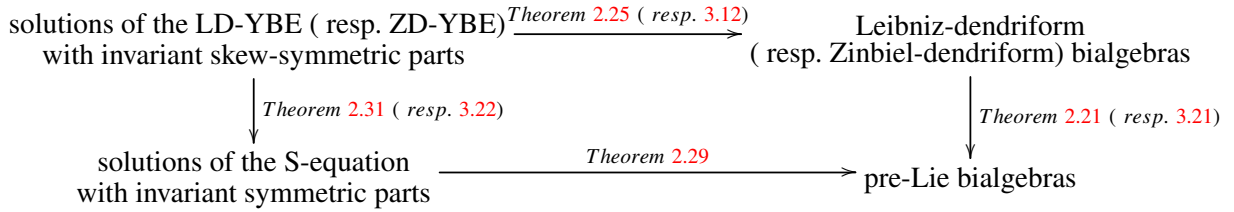
\begin{small}\begin{equation*}
        \xymatrix@C=3cm{
            \txt{ solutions of the LD-YBE ~(~resp. ~ZD-YBE) \\ with invariant skew-symmetric parts} \ar[r]^-{Theorem~\ref{Qy0}~(~resp. ~\ref{Qy1})}
             \ar[d]^-{Theorem ~\ref{Tb4}~(~resp.~ \ref{Tb8})} & \txt{Leibniz-dendriform~\\(~resp.~Zinbiel-dendriform) bialgebras} \ar[d]^-{Theorem~\ref{Tb3}~(~resp.~\ref{Tb7})}\\
            \txt{solutions of the S-equation \\ with invariant symmetric parts} 
            \ar[r]^-{Theorem~\ref{Py}} & \txt{pre-Lie bialgebras}}
    \end{equation*}
\end{small}    

\subsection{ Outline of this paper}    
  
This paper is organized as follows.  
In Section~2, we show that the tensor product of a Leibniz-dendriform algebra 
and a Zinbiel algebra carries a natural
 pre-Lie algebra structure (Proposition~\ref{Tb1}). Dually, a completed pre-Lie coalgebra is induced
 from a Leibniz-dendriform coalgebra by a completed Zinbiel coalgebra (Proposition~\ref{Tb2}). 
 Moreover, We verify that there is a completed pre-Lie bialgebra structure on the tensor product of
  a Leibniz-dendriform bialgebra and a quadratic $\mathbb{Z}$-graded Zinbiel algebra.
  Subsequently, starting from solutions of the LD-YBE with invariant symmetric parts in a Leibniz-dendriform algebra, 
  we obtain completed solutions of the 
$S$-equation whose invariant symmetric parts lie in the resulting pre-Lie algebra. This procedure
 gives rise to quasi-triangular (and triangular) completed pre-Lie bialgebras (Theorem~\ref{Tb4}).
  We further prove that the quasi-triangular pre-Lie bialgebra obtained from a 
  solution of the LD-YBE coincides with the one induced by the corresponding solution 
  of the $S$-equation (Theorem~\ref{Ty1}). Section~3 develops the parallel theory using Zinbiel-dendriform 
  bialgebras and quadratic $\mathbb{Z}$-graded Leibniz algebras, 
  and establishes the affinization characterization of Zinbiel-dendriform bialgebras (Theorem~\ref{Tb7}).

\subsection{ Convention and Notations.} Throughout the paper, $\mathbb K$ is a field.  All vector spaces and algebras are over $\mathbb K$.
 All algebras are finite-dimensional, although many results still hold in the infinite-dimensional case.

 \begin{enumerate}
	\item  Let $A$ be a vector space with a binary operation $\ast:A\otimes A\longrightarrow A$. Define linear maps
$L_{\ast}, R_{\ast}:A\rightarrow \hbox{End}(A)$ by
 $L_{\ast}(a)b:=a\ast b, \  \ R_{\ast}(a)b:=b\ast a$ for all$~a, b\in A$.
 \item Suppose that $V$ is a vector space. Let $\tau:V\otimes V\longrightarrow V\otimes V ,~\tau(w\otimes w')=w'\otimes w$ be the flip operator
 and the identity map $I:V\longrightarrow V, \ \ I(w)=w $ for all$~w, w'\in V$.
 \item Let $V_1, V_2$ be two vector spaces and $\varphi:V_1\longrightarrow V_2 $ be a linear map. Denote the dual (linear)
map by $\varphi^{*}:V_{2}^{*}\longrightarrow V_{1}^{*}$ 
\begin{equation*}\langle u,\varphi^{*}(w^{*})\rangle=\langle \varphi(u),w^{*}\rangle, \ \ ~\forall~u\in V_1,w^{*}\in V_{2}^{*}.\end{equation*}
 \item Let $V$ be a vector space. For any linear map $f:A\longrightarrow \hbox{End}(V)$,
define a linear map $f^{*}:A\longrightarrow \hbox{End}(V^{*})$ by $\langle
f^{*}(x)u^{*},v\rangle=-\langle u^{*},f(x)v\rangle$ for all $x\in A,
u^{*}\in V^{*}, v\in V$.
\item Let
 $r=\sum\limits_{i}a_i\otimes b_i \in A\otimes A$. Put
$ r_{12}=\sum_{i}a_i\otimes b_i\otimes 1,~r_{13}=\sum_{i}a_i\otimes 1\otimes b_i,~r_{32}=\sum_{i} 1\otimes b_i\otimes a_i$,
where $1$ is the unit if $(A,\ast)$ is unital or a symbol playing a similar role to the unit for the
non-unital cases. The operation is given in an obvious way. For example,
\begin{small}
\begin{align*}
r_{12}\ast r_{13}:=\sum_{i,j}a_i\ast a_j\otimes b_i\otimes b_j,\;r_{13}\ast r_{32}:=\sum_{i,j}a_i\otimes b_j\otimes b_i\ast a_j.
\end{align*}\end{small}
\end{enumerate}

\section{Infinite dimensionsal pre-Lie bialgebras from Leibniz-dendriform bialgebras and quadratic $\mathbb{Z}$-graded Zinbiel algebras}
In this section, we show that there is a pre-Lie algebra structure
 on the tensor product of a Leibniz-dendriform algebra 
 and a Zinbiel algebra. Dually, we show that there is a completed pre-Lie coalgebra structure on
  the tensor product of a Leibniz-dendriform coalgebra and a completed Zinbiel coalgebra.
  Moreover, completed pre-Lie bialgebras are obtained through tensor products of Leibniz-dendriform bialgebras and 
  quadratic $\mathbb{Z}$-graded Zinbiel algebras.
  
\subsection{ Pre-Lie algebras induced from Leibniz-dendriform algebras and Zinbiel algebras}

 We begin by recalling some basic results on pre-Lie algebras and
Leibniz-dendriform algebras. 
 
A pre-Lie algebra is a vector space $A$ equipped with a bilinear product $\cdot:A\otimes
A\longrightarrow A$ satisfying
\begin{small}\begin{equation*}(x\cdot y)\cdot z-x\cdot(y\cdot z)=(y \cdot x)\cdot z-y\cdot(x\cdot
z),~\forall~x, y, z \in A.\end{equation*}\end{small}
Denote the associator by $(x,y,z)=(x\cdot y)\cdot z-x\cdot(y\cdot z)$.

 A Leibniz algebra is a vector space $A$ equipped with a binary operation $\circ$ satisfying
\begin{equation*}x\circ(y\circ z)=(x\circ y)\circ z+y\circ(x\circ z), \ \ ~\forall~x,y,z\in A.\end{equation*}

\begin{defi} \cite{TS}
 A {\bf Leibniz-dendriform algebra} is a vector space $A$ together with two binary operations
$\succ,\prec: A\otimes A \rightarrow A$ satisfying
\begin{align} \label{Ld1}
&(x\circ y)\succ z=x\succ(y\succ z)-y\succ(x\succ z),\\
\label{Ld2}&y\prec(x\circ z)+(x\succ y)\prec z=x\succ (y\prec z),\\
\label{Ld3}&x\prec(y\circ z)=(x\prec y)\prec z+y\succ(x\prec z),
\end{align}
for all $x,y,z\in A$, where $x\circ y=x\succ y+x\prec y$.
  $(A,\circ)$ is a Leibniz algebra, which is called the \textbf{sub-adjacent Leibniz algebra} of $(A,\succ,\prec)$
  and $(A,\succ,\prec)$ is called the \textbf{compatible Leibniz-dendriform algebra} on $(A,\circ)$.
\end{defi}

\begin{rmk}
 In \cite{Da}, the Leibniz-dendriform algebra is called the pre-Leibniz algebra.
\end{rmk}

\begin{defi} \cite{L}
	A {\bf (left) Zinbiel algebra} $(A, \diamond)$ is a vector space $A$ together 
with a binary operation $\diamond: A \otimes A \rightarrow A$ satisfying the following identity:
	\begin{equation} \label{Z1}
		a_{1}\diamond(a_{2}\diamond a_{3})=(a_{1}\diamond a_{2})\diamond a_{3}
+(a_{2}\diamond a_{1})\diamond a_{3}, \;\; \forall a_{1}, a_{2}, a_{3}\in A.
	\end{equation}
\end{defi}

\begin{pro}\label{Tb1}
Let $(A, \succ, \prec)$ be a Leibniz-dendriform algebra and $(B, \diamond)$ be a Zinbiel algebra.
Define a binary operation $\cdot$ on $A\otimes B$ by
\begin{small}\begin{align} \label{Pd1}
(a\otimes x)\cdot(b\otimes y)
=a\succ b\otimes x\diamond y
-b\prec a\otimes y\diamond x, 
\end{align}\end{small}
for any $a,b\in A$ and $x,y\in B$. Then
$(A\otimes B, \cdot)$ is a pre-Lie algebra, which is called the {\bf pre-Lie algebra induced 
from $(A, \succ, \prec)$ by $(B, \diamond)$}.  
\end{pro}

\begin{proof}
To prove that $(A\otimes B, \circ)$ is a pre-Lie algebra, we only need to check that
$(a\otimes x,b\otimes y,c\otimes z)=(b\otimes y,a\otimes x,c\otimes z)$ for all $a, b, c\in A$, $x,y,z\in B$.
In fact, using Eq.~\eqref{Pd1} we obtain
\begin{small}\begin{align*}
&(a\otimes x,b\otimes y,c\otimes z)
\\=&(a\succ  b)\succ c\otimes  (x\diamond y)\diamond z-c\prec(a\succ b)\otimes z\diamond (x\diamond y)
\\&- (b\prec a)\succ c\otimes  (y\diamond x)\diamond z+c\prec (b\prec a)\otimes z\diamond (y\diamond x)
-a\succ  (b\succ c)\otimes  x\diamond (y\diamond z)\\&+(b\succ c)\prec a\otimes (y\diamond z)\diamond x 
+ a\succ (c\prec b)\otimes x \diamond(z\diamond y) -(c\prec b)\prec a\otimes (z\diamond y)\diamond x,
\end{align*}\end{small}
and
\begin{small}\begin{align*}
&(b\otimes y,a\otimes x,c\otimes z)
\\=&(b\succ  a)\succ c\otimes  (y\diamond x)\diamond z-c\prec(b\succ a)\otimes z\diamond (y\diamond x)
\\&- (a\prec b)\succ c\otimes  (x\diamond y)\diamond z+c\prec (a\prec b)\otimes z\diamond (x\diamond y)
-b\succ  (a\succ c)\otimes  y\diamond (x\diamond z)\\&+(a\succ c)\prec b\otimes (x\diamond z)\diamond y 
+ b\succ (c\prec a)\otimes y \diamond(z\diamond x) -(c\prec a)\prec b\otimes (z\diamond x)\diamond y.
\end{align*}\end{small}
Combining Eqs.~\eqref{Z1} and ~\eqref{Ld1}-\eqref{Ld3}, we get
\begin{small}\begin{align*}
&(a\otimes x,b\otimes y,c\otimes z)-(b\otimes y,a\otimes x,c\otimes z)
\\=&[(a\succ  b)\succ c+(a\prec  b)\succ c-a\succ  (b\succ c)+b\succ  (a\succ c)]\otimes  x\diamond (y\diamond z)
\\&-[(a\succ  b)\succ c+(a\prec  b)\succ c+(b\prec a)\succ c +(b\succ a)\succ c]\otimes  (y\diamond x)\diamond z
\\&+[a\succ(c\prec  b)-c\prec(a\prec b)-c\prec(a\succ b)-(a\succ  c)\prec b]\otimes  x\diamond (z\diamond y)
\\&+[(a\succ c)\prec  b+(c\prec a)\prec b]\otimes ( z\diamond x)\diamond y-[(b\succ c)\prec a
+(c\prec b)\prec a]\otimes  (z\diamond y)\diamond a
\\&+[(b\succ c)\prec a+c\prec(b\succ  a-b\succ (c\prec a)+c\prec(b\prec a)]\otimes  z\diamond (y\diamond x)
\\=&0.
\end{align*}\end{small}
The proof is completed.
\end{proof}

\subsection{Pre-Lie coalgebras from Leibniz-dendriform coalgebras and completed Zinbiel coalgebras}

In this subsection, we aim to give the dual version of Proposition \ref{Tb1}. 
We begin by reviewing the notions of completed tensor products, mainly following \cite{Ta}.

Let $U=\oplus_{i\in\mathbb{Z}}U_{i}$ and $V=\oplus_{j\in\mathbb{Z}}V_{j}$ be $\mathbb{Z}$-graded vector spaces.
$U\,\hat{\otimes}\,V=\prod_{i,j\in\mathbb{Z}}U_{i}\otimes V_{j}$ is called the {\bf completed tensor product}
 of $U$ and $V$ to be the vector space.
If $U$ and $V$ are finite-dimensional, then $U\,\hat{\otimes}\,V$ coincides with the usual tensor product
 $U\otimes V$. In general, an element of $U\hat{\otimes}V$ is an infinite formal sum $\sum_{i,j\in\mathbb{Z}} X_{ij}$ 
 with $X_{ij}\in U_{i}\otimes V_{j}$. Consequently, each $X_{ij}$ can be expressed as
  $\sum_{\alpha} u_{i,\alpha}\otimes v_{j,\alpha}$, where $u_{i,\alpha}\otimes v_{j,\alpha}\in U_{i}\otimes V_{j}$ 
  are pure tensors and the index set for $\alpha$ is finite.
  With these notations, for linear
maps $\varphi: U\rightarrow U'$ and $\psi: V\rightarrow V'$, we define
\begin{small}\begin{align*}
\varphi\hat{\otimes}\,\psi: U\,\hat{\otimes}\,V\rightarrow U'\,\hat{\otimes}\,V',
\qquad \sum_{i,j,\alpha}u_{i,\alpha}\otimes v_{j, \alpha}\mapsto
\sum_{i,j,\alpha} \varphi(u_{i, \alpha})\otimes \psi(v_{j, \alpha}).
\end{align*}\end{small}
The flip operator $\tau$ can be extended to a completed operator
$\hat{\tau}: V \hat{\otimes} V \rightarrow V \hat{\otimes}V$ given by
$\sum_{i,j,\alpha}u_{i, \alpha}\otimes v_{j, \alpha}\mapsto
\sum_{i,j,\alpha}v_{j, \alpha}\otimes u_{i, \alpha}$.

A \textbf{$\mathbb{Z}$-graded algebra} is an algebra $(A, \ast)$ with 
a linear decomposition $A=\oplus_{i\in\mathbb{Z}}A_{i}$ 
such that $A_{i}\ast A_{j}\subseteq A_{i+j}$ for all $i, j\in\mathbb{Z} $.

\begin{defi} \cite{HL}
A {\bf completed Zinbiel coalgebra} is a pair $(B, \eta)$, where
    $B=\oplus_{i\in\mathbb{Z}}B_{i}$ is a $\mathbb{Z}$-graded vector space
    and $\eta: B\rightarrow B\,\hat{\otimes}\,B$ is a linear map satisfying
\begin{small}\begin{align}\label{Cp1}
    (I\,\hat{\otimes}\,\eta)\eta=(\tau\,\hat{\otimes}\,I)(I\,\hat{\otimes}\,\eta)\eta =(\eta\,\hat{\otimes}\,I)\eta+(\tau\,\hat{\otimes}\,I)(\eta\,\hat{\otimes}\,I)\eta.
    \end{align}\end{small} 
If $B$ is finite-dimensional, then $(B,\eta)$ is called a Zinbiel coalgebra.
\end{defi}

By Eq.~\eqref{Cp1}, we have 
{\small
	\begin{flalign}&\label{Cp2}(I\hat{\otimes}  \eta)\hat{\tau}\eta=(\hat{\tau}\hat{\otimes} I)
(I\hat{\otimes}  \hat{\tau})(\eta\hat{\otimes} I)\eta,
\\&\label{Cp3}
(\eta\hat{\otimes}  I)\hat{\tau}\eta=(I\hat{\otimes}  \hat{\tau})(\hat{\tau}\hat{\otimes} I)(I\hat{\otimes} \eta)\eta,\\
&\label{Cp4}(\hat{\tau}\hat{\otimes} I)(\eta\hat{\otimes} I)\hat{\tau}\eta=(I\hat{\otimes} \eta)\hat{\tau}\eta+
(I\hat{\otimes}  \hat{\tau})(I\hat{\otimes} \eta)\hat{\tau}\eta,
\\&\label{Cp5}
(\hat{\tau}\hat{\otimes} I)(I\hat{\otimes}  \hat{\tau})(I\hat{\otimes} \eta)\hat{\tau}\eta
=(\eta\hat{\otimes} I)\hat{\tau}\eta-(I\hat{\otimes}  \hat{\tau})(\eta\hat{\otimes} I)\eta
\end{flalign}}

\begin{defi} \cite{SG} \label{DB1} A Leibniz-dendriform coalgebra is a triple $(A,\Delta_{\succ},\Delta_{\prec})$, where
$A$ is a vector space and $\Delta_{\succ},\Delta_{\prec}:A\longrightarrow A\otimes A$ are linear maps such that
the following conditions hold:
\begin{align}&\label{Ca1}
( \Delta\otimes I)\Delta_{\succ}=(I\otimes\Delta_{\succ})\Delta_{\succ}-(\tau \otimes I)(I\otimes\Delta_{\succ})\Delta_{\succ},\\&
\label{Ca2}(I\otimes \Delta_{\prec})\Delta_{\succ}-(\Delta_{\succ}\otimes I)\Delta_{\prec}=
(\tau\otimes I)(I \otimes \Delta)\Delta_{\prec},\\&
\label{Ca3}( I\otimes \Delta)\Delta_{\prec}=( \Delta_{\prec}\otimes I)\Delta_{\prec}
+(\tau\otimes I)(I \otimes \Delta_{\prec})\Delta_{\succ},\end{align}
where $\Delta=\Delta_{\succ}+\Delta_{\prec}$.
\end{defi}

A {\bf complete pre-Lie coalgebra} \cite{LZB}
 is a pair $(L, \delta)$, where
$L=\oplus_{i\in \mathbb{Z}} L_i$
  is a $\mathbb{Z}$-graded vector space and $\delta: L\rightarrow L\hat{\otimes} L$ is a linear map satisfying
    \begin{align}\label{Pl1}
    (I\hat{\otimes} \delta)\delta(a)-(\tau\hat{\otimes} I)(I\hat{\otimes} \delta)\delta(a)=(\delta\hat{\otimes} I\delta(a)
    -(\tau\hat{\otimes} I)(\delta\hat{\otimes} I)\delta(a),  \;\; a\in L. \end{align}
If $L$ is finite-dimensional, then $(L,\delta)$ is called a Zinbiel coalgebra.   
    
We now state the dual version of Proposition \ref{Tb1}.

 \begin{pro} \label{Tb2}
Let $(A, \Delta_{\succ},\Delta_{\prec})$ be a Leibniz-dendriform coalgebra and
$(B,\eta)$ be a completed Zinbiel coalgebra.
 Define a linear map $\delta: A\otimes B\rightarrow(A\otimes B)\otimes(A\otimes B)$ by
\begin{small}\begin{align}
\delta(a\otimes x)
&=\Delta_{\succ}(a)\bullet \eta(x)-\tau\Delta_{\prec}(a)\bullet \tau\eta(x)
 \label{P1} \\
&=\sum_{(a)}\sum_{i,j,\alpha}\Big((a_{1}\otimes x_{1,i,\alpha})
\otimes(a_{2}\otimes x_{2,j,\alpha})-(a_{(2)}\otimes x_{2,j,\alpha})
\otimes(a_{(1)}\otimes x_{1,i,\alpha})\Big),  \nonumber
\end{align}\end{small}
for any $a,b\in A$ and $x,y\in B$, where
 $\Delta_{\succ}(a)=\sum_{(a)}a_{1}\otimes a_{2},~\Delta_{\prec}(a)=\sum_{(a)}a_{(1)}\otimes a_{(2)}$
 and  $\eta(x)=\sum_{i,j,\alpha}x_{1,i,\alpha}\otimes x_{2,j,\alpha}$.
 Then $(A\otimes B, \delta)$ is a completed pre-Lie coalgebra, which is called the \textbf{completed pre-Lie coalgebra
} induced from $(A,\Delta_{\succ},\Delta_{\prec})$ by $(B, \eta)$.
\end{pro}

\begin{proof}By Eq.~\eqref{P1}, we have
{\small
	\begin{flalign}
		&\label{P2}(\delta\hat{\otimes} I)\delta(a\otimes x)-(I\hat{\otimes} \delta)\delta(a\otimes x)
-(\hat{\tau}\hat{\otimes} I)(\delta\hat{\otimes} I)\delta(a\otimes x)
+(\hat{\tau}\hat{\otimes} I)(I\hat{\otimes} \delta)\delta(a\otimes x)
\\=&\nonumber(\Delta_{\succ}\otimes I)\Delta_{\succ}(a)\bullet (\eta \hat{\otimes} I)\eta (x)
-(\tau\otimes I)(\Delta_{\prec}\otimes I)\Delta_{\succ}(a)\bullet
 (\hat{\tau}\hat{\otimes} I)(\eta \hat{\otimes} I)\eta (x)
\\&\nonumber-(\Delta_{\succ}\otimes I)\tau\Delta_{\prec}(a)\bullet
 (\eta \hat{\otimes} I)\hat{\tau}\eta (x)
+(\tau\otimes I)(\Delta_{\prec}\otimes I)\tau\Delta_{\prec}(a)\bullet
 (\hat{\tau}\hat{\otimes} I)(\eta \hat{\otimes} I)\hat{\tau}\eta (x),
\\&\nonumber-(I\otimes \Delta_{\succ})\Delta_{\succ}(a)\bullet (I \hat{\otimes} \eta)\eta (x)
+(I\otimes \tau)(I\otimes \Delta_{\prec})\Delta_{\succ}(a)\bullet (I\hat{\otimes}  \hat{\tau})(I\hat{\otimes} \eta )\eta (x)
\\&\nonumber+(\tau\otimes I)(I\otimes \tau)(\Delta_{\succ}\otimes I)\Delta_{\prec}(a)
\bullet (\hat{\tau}\hat{\otimes} I)(I\hat{\otimes} \hat{\tau})(\eta \hat{\otimes} I)\eta (x)
-(I\otimes \tau)(I\otimes \Delta_{\prec})\tau\Delta_{\prec}(a)\bullet
 (I\hat{\otimes} \hat{\tau})(I \hat{\otimes} \eta)\hat{\tau}\eta (x)
\\&\nonumber-(\tau\otimes I)(\Delta_{\succ}\otimes I)\Delta_{\succ}(a)\bullet
(\hat{\tau}\hat{\otimes} I) (\eta \hat{\otimes} I)\eta (x)
+(\Delta_{\prec}\otimes I)\Delta_{\succ}(a)\bullet (\eta \hat{\otimes} I)\eta (x)
\\&\nonumber+(\tau\otimes I)(\Delta_{\succ}\otimes I)\tau\Delta_{\prec}(a)
\bullet (\hat{\tau}\hat{\otimes} I) (\eta \hat{\otimes} I)\hat{\tau}\eta (x)
-(\Delta_{\prec}\otimes I)\tau\Delta_{\prec}(a)\bullet (\eta \hat{\otimes} I)\hat{\tau}\eta (x)
\\&\nonumber+(\tau\otimes I) (I\otimes \Delta_{\succ})\Delta_{\succ}(a)
\bullet (\hat{\tau}\hat{\otimes} I)(I \hat{\otimes} \eta)\eta (x)
-(\tau\otimes I)(I\otimes \tau)(I\otimes \Delta_{\prec})\Delta_{\succ}(a)\bullet 
(\hat{\tau}\hat{\otimes} I)(I\hat{\otimes}  \hat{\tau})(I\hat{\otimes} \eta )\eta (x)
\\&\nonumber-(I\otimes \tau)(\Delta_{\succ}\otimes I)\Delta_{\prec}(a)\bullet
 (I\hat{\otimes} \hat{\tau})(\eta \hat{\otimes} I)\eta (x)
+(\tau\otimes I)(I\otimes \tau)(I\otimes \Delta_{\prec})\tau\Delta_{\prec}(a)\bullet
 (\hat{\tau}\hat{\otimes} I)(I\hat{\otimes} \hat{\tau})(I \hat{\otimes} \eta)\hat{\tau}\eta (x)\nonumber.
\end{flalign}}
According to Eqs.~\eqref{Cp1}-\eqref{Cp5}, we obtain that
{\small
	\begin{flalign*}
		&(\delta\hat{\otimes} I)\delta(a\otimes x)-(I\hat{\otimes} \delta)\delta(a\otimes x)
-(\hat{\tau}\hat{\otimes} I)(\delta\hat{\otimes} I)\delta(a\otimes x)
+(\hat{\tau}\hat{\otimes} I)(I\hat{\otimes} \delta)\delta(a\otimes x)
\\=&A(1)+A(2)+A(3)+A(4)+A(5)+A(6),
\end{flalign*}}
where
{\small
	\begin{flalign*}
		&A(1)=[(\Delta\otimes I)\Delta_{\succ}(a)
-(I\otimes \Delta_{\succ})\Delta_{\succ}(a)+(I\otimes \tau)(I\otimes \Delta_{\succ})\Delta_{\succ}(a)]\bullet 
(I \hat{\otimes} \eta)\eta (x),
\\&A(2)=(\tau\otimes I)[(\Delta\otimes I)\tau \Delta_{\prec}(a)
+(I\otimes \tau)(\Delta_{\succ}\otimes I)\Delta_{\prec}(a)-(I\otimes \tau)(I\otimes \Delta_{\prec})\Delta_{\prec}(a)]
\bullet (I \hat{\otimes} \eta)\hat{\tau}\eta (x),
\\&A(3)=-[(\tau\otimes I)(I\otimes \tau)(I\otimes \Delta_{\prec})\tau\Delta_{\prec}(a)+(I\otimes \tau)(\Delta_{\succ}\otimes I)\Delta_{\prec}(a)]
\bullet (I \hat{\otimes} \hat{\tau})(\eta\hat{\otimes} I)\eta (x),\\&
A(4)=[(I\otimes \tau)(I\otimes \Delta_{\prec})\Delta_{\succ}(a)-(\Delta\otimes I)\tau\Delta_{\prec}(a)+
(\tau\otimes I)(I\otimes \tau)(I\otimes \Delta_{\prec})\tau\Delta_{\prec}(a)]\bullet(\eta\hat{\otimes} I)\hat{\tau}\eta (x)
\\&A(5)=-[(\Delta\otimes I)\Delta_{\succ}(a)+(\tau\otimes I)(\Delta\otimes I)\Delta_{\succ}(a)]
\bullet(\hat{\tau}\hat{\otimes} I)(\eta\hat{\otimes} I)\eta (x),
\\&A(6)=[(\tau\otimes I)(\Delta\otimes I)\tau\Delta_{\prec}(a)-
(I\otimes \tau)(I\otimes \Delta_{\prec})\tau\Delta_{\prec}(a)-
(\tau\otimes I)(I\otimes \tau)(I\otimes \Delta_{\prec})\Delta_{\succ}(a)]
\bullet(I\hat{\otimes} \hat{\tau})(I\hat{\otimes} \eta)\hat{\tau}\eta (x).
\end{flalign*}}
Then, 
A(1)=0 follows from  Eq.~\eqref{Ca1}, A(2)= A(5)=0 follows from Eq.~\eqref{Ca2}, 
A(3)=A(4)=0 follows from Eq.~\eqref{Ca2}-\eqref{Ca3}, and A(6)=0 follows from  Eq.~\eqref{Ca3}.
The proof is completed.
\end{proof}

\begin{rmk} 
Propositions~\ref{Tb1} and~\ref{Tb2} remain valid when Zinbiel (co)algebras are 
replaced by Zinbiel (co)dialgebras. However, since Zinbiel (co)dialgebras are more intricate, 
whereas Zinbiel (co)algebras are precisely their commutative counterparts,
 we restrict our discussion to the (co)algebra setting for convenience.
 
\end{rmk}

\subsection{Pre-Lie bialgebras from Leibniz-dendriform bialgebras and quadratic $\mathbb{Z}$-graded Zinbiel algebras}

We establish a pre-Lie (co)algebra structure on the tensor product of a
 Leibniz-dendriform (co)algebra and a Zinbiel (co)algebra (Propositions \ref{Tb1} and \ref{Tb2}). 
 This naturally raises the question of constructing (completed) pre-Lie bialgebras from 
 Leibniz-dendriform bialgebras and Zinbiel algebras, which is the main focus of this subsection.
 
 First, we recall some basic notions of Leibniz-dendriform bialgebras and (completed) pre-Lie bialgebras.

\begin{defi} \cite {SG} A Leibniz-dendriform bialgebra is a quintuple $(A,\succ,\prec,\Delta_{\succ},\Delta_{\prec})$
such that $(A,$ \ \ \  $\succ,\prec)$ is a Leibniz-dendriform algebra, $(A,\Delta_{\succ},\Delta_{\prec})$ is
a Leibniz-dendriform coalgebra and the following compatible conditions hold:
\begin{small}
\begin{align}&\label{B1}
\Delta(x\odot y)-(I\otimes L_{\odot}(x))\Delta(y)+\tau(I\otimes L_{\odot}(x))\Delta(y)+
\tau (I\otimes R_{\odot}(y))\Delta_{\succ}(x)-(I\otimes R_{\odot}(y))\Delta_{\succ}(x)=0,\\&
\label{B2}\Delta(x\succ y)-(L_{\succ}(x)\otimes I+I\otimes L_{\succ}(x))\Delta (y)-(I\otimes R_{\succ}(y))\Delta_{\odot}(x)
+\tau(I\otimes R_{\odot}(y))\Delta_{\odot}(x)=0,\\&
\label{B3}(I\otimes R_{\succ}(y))\tau\Delta_{\prec}(x)-(R_{\prec}(x)\otimes I)\Delta(y)=0,\\&
\label{B4}(\Delta_{\succ} +\tau\Delta_{\prec})(x\circ y)-(I\otimes L_{\circ}(x)+L_{\succ}(x)\otimes I)\Delta_{\odot}(y)
+(I\otimes L_{\circ}(y)+L_{\succ}(y)\otimes I)\Delta_{\odot}(x)=0,\\&
\label{B5}\Delta_{\succ}(x\circ y)-(I\otimes R_{\circ}(y))\Delta_{\succ}(x)-(I\otimes L_{\circ}(x)+L_{\odot}(x)\otimes I)\Delta_{\succ}(y)
+(L_{\odot}(y)\otimes I)\Delta_{\odot}(x)=0,\\&
\label{B6}(I\otimes R_{\circ}(y))\tau\Delta_{\prec}(x)
-(R_{\prec}(x)\otimes I)\Delta_{\succ}(y)=0,
\end{align}
\end{small}
where $x\circ y=x \succ y+x\prec y,~x\odot y=x\succ y+y\prec x,~\Delta=\Delta_{\succ}+\Delta_{\prec},
 ~\Delta_{\odot}=\Delta_{\succ}+\tau\Delta_{\prec}
$ and $R_{\circ}=R_{\prec}+R_{\succ},
~L_{\circ}=L_{\prec}+L_{\succ},~L_{\odot}=R_{\prec}+L_{\succ},~R_{\odot}=L_{\prec}+R_{\succ}$.
\end{defi}

\begin{defi}\cite{TXS}
A quadratic Leibniz-dendriform algebra is a
Leibniz-dendriform algebra $(A,\succ,\prec)$ equipped with a non-degenerate symmetric bilinear
form $\omega_{A}\in A^{*}\otimes A^{*}$ such that the following
invariant conditions hold for all $x, y, z \in A$,
\begin{equation} \label{C1}\omega_{A} (x \prec y, z)=\omega_{A}(x, y\circ z+z\circ y), \ \  \
\omega_{A}(x \succ y, z)=-\omega_{A}( y,x\circ z),\end{equation}
where $x \circ y=x \succ y+x \prec y$.
\end{defi}

\begin{defi}\cite{TXS}
A Manin triple of Leibniz-dendriform algebras is a triple
of Leibniz-dendriform algebras $(A,A_{1},A_{2})$ satisfying
\begin{enumerate}
\item $A=A_{1}\oplus A_{2}$ as vector spaces,
\item $(A,\succ,\prec,\omega)$ is a quadratic Leibniz-dendriform algebra,
\item $A_{1},A_{2}$ are isotropic subalgebras of $A$, that is,
$\omega(x,y) =\omega(a,b) =0$ for all $x,y\in A_1$ and $a,b\in A_2$.
\end{enumerate}
\end{defi}

\begin{thm} \cite{SG} \label{Be}
Let $(A,\succ_A,\prec_A)$ be a Leibniz-dendriform algebra and $(A,\circ )$ be the sub-adjacent
 Leibniz algebra of $(A,\succ_{A},\prec_{A})$.
Suppose that there is a Leibniz-dendriform algebra $(A^{*}, \succ_{A^*},\prec_{A^*})$ which is induced from
a Leibniz-dendriform coalgebra $(A, \Delta_{\succ}, \Delta_{\prec})$. Then
$(A,\succ_A,\prec_A,\Delta_{\succ},\Delta_{\prec})$ is a Leibniz-dendriform bialgebra if and only if
 $(A\oplus A^{*},A, A^{*},\omega)$ is a Manin triple of Leibniz-dendriform algebras with
the bilinear form $\omega$ given by:
  \begin{equation}\label{C2}\omega(x+\zeta, y+\eta) =\langle x,\eta\rangle+\langle \zeta,y\rangle,  
  \ \ ~\forall~x, y\in A,~\zeta,\eta\in A^{*}.\end{equation}
\end{thm}

\begin{defi} \cite{LH,B}
Let $(L, \cdot)$ be a pre-Lie algebra and $(L,
\delta)$ be a {\bf completed }pre-Lie coalgebra. If the following
compatibility conditions are satisfied for all $ a, b\in L$:
\begin{align}
&\label{Pb1}(\delta-\hat{\tau}\delta)(a\cdot b)-(L_{\cdot}(a)\hat{\otimes} I)(\delta-\hat{\tau}\delta)(b)
-(I\hat{\otimes} L_{\cdot}(a))(\delta-\hat{\tau}\delta)(b)
\\&-(I\hat{\otimes} R_{\cdot}(b))\delta(a)+(R_{\cdot}(b)\hat{\otimes}I)\hat{\tau}\delta(a)=0, \nonumber\\
&\label{Pb2}\delta(a\cdot b-b\cdot a)-(I\hat{\otimes} (R_{\cdot}(b)-L_{\cdot}(b)))\delta(a)-(I\hat{\otimes}(L_{\cdot}(a)-R_{\cdot}(a)))\delta(b)
\\&-(L_{\cdot}(a)\hat{\otimes}I)\delta(b)+(L_{\cdot}(b)\hat{\otimes} I)\delta(a)=0,\nonumber
\end{align}
then we call $(L, \cdot,\delta)$ a {\bf completed pre-Lie bialgebra}. 
If $L$ is finite-dimensional,  $(L, \cdot,\delta)$ is just the usual pre-Lie bialgebra.
\end{defi}

Recall that
a Lie algebra $(\mathfrak{g},[ \ , \ ])$ is called a \textbf{symplectic Lie algebra} \cite{B}
  if there is a nondegenerate skew-symmetric
  bilinear form $\omega_p$ (called a 2-cocycle) on $ \mathfrak{g} $, that is,
	\begin{align}
		\omega_p ([x,y],z)+\omega_p ([y,z],x)+\omega_p ([z,x],y)=0 ,\quad x,y,z\in \mathfrak{g}.
	\end{align}
	Denote it by $(\mathfrak{g}, [ \ , \ ],\omega_p)$.
	
	A symplectic Lie algebra $(\mathfrak{g}, [ \ , \ ],\omega_p)$ is called a {\bf para-K$\ddot{\texttt{a}}$hler Lie algebra} \cite{B} if $\mathfrak{g}=\mathfrak{g}_1\oplus\mathfrak{g}_2$ is the direct sum of the underlying vector spaces of two Lie subalgebras in which $\omega_p(\mathfrak{g}_i,\mathfrak{g}_i)=0$ for $i=1,2$. Denote it by $(\mathfrak{g}_1\bowtie\mathfrak{g}_2,\mathfrak{g}_1,\mathfrak{g}_2,\omega_p)$.

\begin{pro}\cite{Ch}
	Let	$(\mathfrak{g},[ \ , \ ],\omega_p)$ be a symplectic Lie algebra. Then there is  a compatible pre-Lie algebra
 structure $\cdot$ on $\mathfrak{g}$ given by
	\begin{align}\label{Sp}
		\omega_p(x\cdot y,z)=-\omega_p(y,[x,z]),\quad x,y,z\in \mathfrak{g}.
	\end{align}
\end{pro}

\begin{defi}\cite{LZB}
Let $(L, \cdot)$ be a pre-Lie algebra. If  there is a non-degenerate skew-symmetric bilinear 
form $\omega_p$ on $L$ such that Eq.~(\ref{Sp}) holds, where $[x,z]=x\cdot z-z\cdot x$ 
for all $x$, $z\in L$, then $(L, \cdot, \omega_p)$ is called a {\bf quadratic pre-Lie algebra}.
\end{defi}

\begin{pro}\cite{B}
	Let	$(A,\cdot_A)$ be a pre-Lie algebra. Suppose that $(A^*, \cdot_{A^*})$ is a pre-Lie algebra 
which is induced from a pre-Lie coalgebra $(A,\delta)$.  Then $(A,\cdot ,\delta)$ is a pre-Lie bialgebra 
if and only if $(\mathfrak{g}(A)\bowtie \mathfrak{g}(A^*),\mathfrak{g}(A),\mathfrak{g}(A^*),\omega_p)$ 
is a para-K$\ddot{a}$hler Lie algebra, where $\omega_p$ is given by
\begin{eqnarray*}
\omega_p(a+f,b+g)=\langle f,b\rangle -\langle g,a\rangle,\quad  a,b\in A,\;\;f, g\in A^*.
\end{eqnarray*}
\end{pro}
 
 Building on the above works, we study the following proposition, which is vital.
 
\begin{pro}\label{Sp2}
Let $(A, \succ, \prec,\rho)$ be a quadratic Leibniz-dendriform algebra, $(B, \diamond)$ be 
a Zinbiel algebra with a non-degenerate bilinear form $\nu$, and $(A\otimes B, \cdot)$ be the induced pre-Lie algebra from
$(A, \succ, \prec,\rho)$ by $(B, \diamond)$. Then $(A\otimes B, \cdot, \omega)$ is a quadratic pre-Lie algebra with $\omega$ defined by
\begin{align*}
\omega(a\otimes x,b\otimes y)=\rho(a,b)\nu(x, y),\;\;\;a, b\in A,\;\;x, y\in B,
\end{align*}
if and only if $\nu$ is skew-symmetric and the following equality holds for all $a,b,c\in A$ and $x,y,z\in B$:
\begin{align*}
\rho( b,a\circ c)\Big(\nu(y,x\diamond z)-\nu(x\diamond y,z)-\nu(y\diamond x, z)\Big)
-\rho(b,c\circ a)\Big(\nu(y\diamond x,z)+\nu(y,z\diamond x)\Big)=0.
\end{align*}
\end{pro}
\begin{proof} By direct computations, we have for all $a, b,c\in A,\;\;x, y,z\in B$,
\begin{align*}
	&\omega((a\otimes x)\cdot (b\otimes y),c\otimes z)+\omega(b\otimes y,[a\otimes x,c\otimes z ])
\\=&\rho(a\succ b,c)\nu(x\diamond y,z)-\rho(b\prec a,c)\nu(y\diamond x,z)
	+\rho(b,a\succ c)\nu(y,x\diamond z)\\&-\rho(b,c\prec a)\nu(y,z\diamond x)
	-\rho(b,c\succ a)\nu(y,z\diamond x)+\rho(b,a\prec c)\nu(y,x\diamond z)
\\=&-\rho( b,a\circ c)\nu(x\diamond y,z)-\rho(b,c\circ a+a\circ c)\nu(y\diamond x,z)
\\&+\rho(b,a\circ c)\nu(y,x\diamond z)-\rho(b,c\circ a)\nu(y,z\diamond x)
\\=&-\rho( b,a\circ c)\Big(\upsilon(x\diamond y,z)+\nu(y\diamond x, z)-\nu(y,x\diamond z)\Big)
-\rho(b,c\circ a)\Big(\nu(y\diamond x,z)+\nu(y,z\diamond x)\Big)
\end{align*}
This finishes the proof.
\end{proof}

 Recall that a {\bf quadratic Zinbiel algebra} \cite{B2, ZLS} is a triple $(B,\diamond, \nu)$, where 
$(B, \diamond)$ is a Zinbiel algebra and $\nu$ is a nondegenerate skew-symmetric bilinear form on $A$ satisfying the invariance condition
\begin{equation*}
	\nu(a_{1}\diamond a_{2}, \; a_{3}) = \nu(a_{2}, \; a_{1}\diamond a_{3} + a_{3}\diamond a_{1}), \;\; \forall~ a_{1}, a_{2}, a_{3} \in B.
\end{equation*}
If $(B, \diamond, \nu)$ is a quadratic Zinbiel algebra, 
then one easily checks that 
\begin{equation*}\nu(a_{1} \diamond a_{2},\; a_{3}) = -\nu(a_{1},\; a_{3} \diamond a_{2}).\end{equation*} 

By Proposition \ref{Sp2}, if the non-degenerate skew-symmetric bilinear form $\nu$ satisfies the condition
$\nu(x\diamond y,z)+\nu(y\diamond x, z)-\nu(y,x\diamond z)$=0, that is, $(B, \diamond,\nu)$ is a quadratic Zinbiel algebra,
then $(A\otimes B, \cdot, \omega)$ is a quadratic pre-Lie algebra.
Consequently, we study how to construct (completed) pre-Lie bialgebras from 
Leibniz-dendriform bialgebras and quadratic Zinbiel algebras.

\begin{defi}\cite{LZB} \label{Tp6}
Let $(B=\oplus_{i\in\mathbb{Z}}B_{i}, \diamond)$ be a $\mathbb{Z}$-graded (Zinbiel) algebra. A bilinear form $\omega$ on $B$
 is called {\bf graded}, if there exists some $m\in\mathbb{Z}$ such that $\omega(B_{i}, B_{j})=0$ when $i+j+m\neq0$.
A {\bf quadratic $\mathbb{Z}$-graded Zinbiel algebra}, denoted by $(B = \bigoplus_{i \in \mathbb{Z}} B_i, \diamond, \omega)$, 
is a $\mathbb{Z}$-graded Zinbiel algebra equipped with a nondegenerate skew-symmetric,
invariant and graded bilinear form $\omega$. In the special case where $B = B_0$, it reduces to a quadratic Zinbiel algebra.
\end{defi}

Given a quadratic $\mathbb{Z}$-graded (Zinbiel) algebra $(B=\oplus_{i\in\mathbb{Z}}B_{i}, \diamond, \omega)$.
Define bilinear forms
\begin{small}\begin{align*}\hat{\omega}:
(\underbrace{B\,\hat{\otimes}\,B\,\hat{\otimes}\,\cdots\,\hat{\otimes}\,
B}_{n\text{-fold}})\otimes(\underbrace{B\otimes B\otimes\cdots
\otimes B}_{n\text{-fold}})\longrightarrow \mathbb K,
\end{align*}\end{small}
for all $n\geq2$ by
\begin{small}\begin{align*}
\hat{\omega}\Big(\sum_{i_{1},\cdots,i_{n},\alpha} x_{1, i_{1}, \alpha}
\otimes\cdots\otimes x_{n, i_{n}, \alpha},\ \ y_{1}\otimes\cdots\otimes y_{n}\Big)
=\sum_{i_{1},\cdots,i_{n},\alpha}\prod_{j=1}^{n}
\omega(x_{j, i_{j}, \alpha},\; y_{j}).
\end{align*}\end{small}
It is easy to prove that $\hat{\omega}$ is left nondegenerate, that is, if
\begin{small}\begin{align*}
\hat{\omega}\Big(\sum_{i_{1}, \cdots, i_{n},\alpha}x_{1, i_{1}, \alpha}
\otimes\cdots\otimes x_{n, i_{n}, \alpha},\ \ y_{1}\otimes\cdots\otimes y_{n}\Big)
=\hat{\omega}\Big(\sum_{i_{1},\cdots,i_{n},\alpha} z_{1, i_{1}, \alpha}
\otimes\cdots\otimes z_{n, i_{n}, \alpha},\ \ y_{1}\otimes\cdots\otimes y_{n}\Big),
\end{align*}\end{small}
for all homogeneous elements $y_{1}, y_{2},\cdots, y_{n}\in B$, then
\begin{align*}
\sum_{i_{1},\cdots,i_{n},\alpha} x_{1, i_{1}, \alpha}
\otimes\cdots\otimes x_{n, i_{n}, \alpha}
=\sum_{i_{1},\cdots,i_{n},\alpha} z_{1, i_{1}, \alpha}
\otimes\cdots\otimes z_{n, i_{n}, \alpha}.
\end{align*}

\begin{lem} \cite{HL} Let $(B=\oplus_{i\in\mathbb{Z}}B_{i}, \diamond, \omega)$ be a quadratic $\mathbb{Z}$-graded Zinbiel algebra.
Define a linear map $\eta_{\omega}: B\rightarrow B\otimes B$ by $\hat{\omega}(\eta_{\omega}
(b_{1}),\; b_{2}\otimes b_{3})=-\omega(b_{1},\; b_{2}\diamond b_{3})$ for any $b_{1}, b_{2},
b_{3}\in B$. Then $(B, \eta_{\omega})$ is a completed Zinbiel coalgebra. \end{lem}

Assume that $(B=\oplus_{i\in\mathbb{Z}}B_{i}, \diamond,\omega)$ is a quadratic $\mathbb{Z}$-graded Zinbiel algebra.
Since $\hat{\omega}$ is left nondegenerate, it follows that
\begin{align}
&\label{Bg1}(I\otimes L_{\diamond}(x))\eta_{\omega}(y)=\eta_{\omega}(x\diamond y)-\hat{\tau}\eta_{\omega}(y\diamond x),
\\&\label{Bg2}(I\otimes R_{\diamond}(x))\eta_{\omega}(y)=\eta_{\omega}(y\diamond x)+\hat{\tau}\eta_{\omega}(y\diamond x),\\&
\label{Bg3}(I\otimes L_{\diamond}(x))\eta_{\omega}(y)=\eta_{\omega}(y\diamond x)-\hat{\tau}\eta_{\omega}(x\diamond y),\\&
\label{Bg4}(L_{\diamond}(x)\otimes I )\eta_{\omega}(y)
=(L_{\diamond}(y)\otimes I )\eta_{\omega}(x)=\eta_{\omega}(x\diamond y)+\eta_{\omega}(y\diamond x),\\&
\label{Bg5}(R_{\diamond}(y)\otimes I )\eta_{\omega}(x)
=-\eta_{\omega}(x\diamond y)-\eta_{\omega}(y\diamond x)-\hat{\tau}\eta_{\omega}(x\diamond y)
-\hat{\tau}\eta_{\omega}(y\diamond x)-\hat{\tau}(R_{\diamond}(x)\otimes I )\eta_{\omega}(y).\end{align}
One can show that the above equations hold; alternatively, the proof is given in \cite{HL}.

\begin{ex} \label{Zq2}
Assume that \[V=\operatorname{span}_{\mathbb K}\{v_1,v_2,v_3,v_1^{*},v_2^{*},v_3^{*}\}\] is the $6$-dimensional vector space.
Define a binary operation $\diamond$ on $A$ by 
\begin{align*}
&v_1\diamond v_1=v_2,~~~v_1\diamond v_2=v_3,~~~~v_2\diamond v_1=\frac{1}{2}v_3,\\&
v_1\diamond v_2^{*}=2v_1^{*},~~~~v_1\diamond v_3^{*}=\frac{3}{2}v_2^{*},~~~v_2\diamond v_3^{*}=\frac{3}{2}v_1^{*}\\&    
v_2^{*}\diamond v_1=-v_1^{*},~~~v_3^{*}\diamond v_1=-\frac{1}{2}v_2^{*},~~~ v_3^{*}\diamond v_2=-v_1^{*}.
\end{align*}
and all other products equal to zero. By direct computations, $(V,\diamond)$ is a (left) Zinbiel algebra.
Moreover, define a skew-symmetric bilinear form $\epsilon$ on $V$ by
$\epsilon(v_i,v_i^{*})=1$ with the remaining values determined by skew-symmetry. Then
$(V,\diamond,\epsilon)$ is a quadratic Zinbiel algebra.
Let 
\[\widehat V=V\otimes \mathbb K[t,t^{-1}]=\bigoplus_{n\in\mathbb Z}(V\otimes t^n)\]
equipped with \[(x\otimes t^m)\diamond(y\otimes t^n)=(x\diamond y)\otimes t^{m+n}\]
and \[\omega(x\otimes t^m,y\otimes t^n)=\epsilon(x,y)\delta_{m+n,0}\]
then $(V,\diamond,\omega)$ is a quadratic $\mathbb Z$-graded Zinbiel algebra.
The induced comultiplication is given by~(denote $xt^m=x\otimes t^m,~~\forall~x\in V$),
\[
\begin{aligned}
\eta_{\omega}(v_1t^n) &=\sum_{i+j=n} 2\, v_1^* t^i\otimes v_2t^j 
+ \frac{3}{2}\, v_2^* t^i\otimes v_3t^j - v_2t^i \otimes v_1^*t^j - v_3t^i \otimes v_2^*t^j, \\
\eta_{\omega}(v_2t^n) &=\sum_{i+j=n} \frac{3}{2}\, v_1^* t^i\otimes v_3t^j - \frac{1}{2}\, v_3t^i \otimes v_1^*t^j, \\
\eta_{\omega}(v_3t^n) &= \eta_{\omega}(v_1^*t^n) = 0, \\
\eta_{\omega}(v_2^*t^n) &=\sum_{i+j=n} v_1^*t^i \otimes v_1^*t^j, \\
\eta_{\omega}(v_3^*t^n) &=\sum_{i+j=n} v_1^*t^i \otimes v_2^*t^j + \frac{1}{2}\, v_2^* t^i\otimes v_1^*t^j.
\end{aligned}
\]
\end{ex}

We lift Proposition \ref{Tb1} and Proposition \ref{Tb2} to the bialgebra setting.

\begin{thm}\label{Tb3}
Let $(A, \succ,\prec,\Delta_{\succ}, \Delta_{\prec})$ be a Leibniz-dendriform bialgebra
and $(B=\oplus_{i\in\mathbb{Z}}B_{i}, \diamond,\omega)$ be a quadratic $\mathbb{Z}$-graded Zinbiel algebra.
Define a binary operation $\cdot$ on $A\otimes B$ by Eq.~\eqref{Pd1}, and
a linear map $\delta: A\otimes B\rightarrow(A\otimes B)
\,\hat{\otimes}\,(A\otimes B)$ by Eq.~\eqref{P1}.
Then $(A\otimes B,\cdot,\delta)$ is a completed pre-Lie bialgebra.
\end{thm} 

 \begin{proof}
According to Proposition \ref{Tb1} and Proposition  \ref{Tb2}, we only need
to check that Eqs.~\eqref{Pb1}-\eqref{Pb2} hold. Using Eqs.~\eqref{P1} and \eqref{Pd1}, we have
\begin{small} \begin{align*}&
(\delta-\hat{\tau}\delta)((a\otimes x)\cdot (b\otimes y))-(L_{\cdot}(a\otimes x)\hat{\otimes} I)(\delta
-\hat{\tau}\delta)(b\otimes y)-(I\hat{\otimes} L_{\cdot}(a\otimes x))(\delta-\hat{\tau}\delta)(b\otimes y)
\\&-(I\hat{\otimes} R_{\cdot}(b\otimes y))\delta(a\otimes x)
 -( R_{\cdot}(b\otimes y)\hat{\otimes} I)\hat{\tau}\delta(a\otimes x)
\\=&\Delta(a\succ b)\bullet \eta_{\omega} (x\diamond y)-\Delta(b\prec a)\bullet \eta_{\omega} (y\diamond x)
-\tau\Delta(a\succ b)\bullet \hat{\tau}\eta_{\omega} (x\diamond y)+\tau\Delta(b\prec a)\bullet \hat{\tau}\eta_{\omega} (y\diamond x)\\&
+(L_{\succ}(a)\otimes I)\tau\Delta(b)\bullet (L_{\diamond}(x)\otimes I)\hat{\tau}\eta_{\omega} (y)
 -(L_{\succ}(a)\otimes I)\Delta(b)\bullet (L_{\diamond}(x)\otimes I)\eta_{\omega} (y)
 \\&+(R_{\prec}(a)\otimes I)\Delta(b)\bullet (R_{\diamond}(x)\otimes I)\eta_{\omega} (y)
 - (R_{\prec}(a)\otimes I)\tau\Delta(b)\bullet (R_{\diamond}(x)\otimes I)\hat{\tau}\eta_{\omega} (y)
 \\&+(I\otimes L_{\succ}(a))\tau\Delta(b)\bullet (I\otimes L_{\diamond}(x))\hat{\tau}\eta_{\omega} (y)
 -(I\otimes L_{\succ}(a))\Delta(b)\bullet (I\otimes L_{\diamond}(x))\eta_{\omega} (y)
 \\&+(I\otimes R_{\prec}(a))\Delta(b)\bullet (I\otimes R_{\diamond}(x))\eta_{\omega} (y)-
 (I\otimes R_{\prec}(a))\tau\Delta(b)\bullet (I\otimes R_{\diamond}(x))\hat{\tau}\eta_{\omega} (y)\\&
 -(I\otimes R_{\succ}(b))\Delta_{\succ}(a)\bullet (I\otimes R_{\diamond}(y))\eta_{\omega} (x)
 +(I\otimes L_{\prec}(b))\Delta_{\succ}(a)\bullet (I\otimes L_{\diamond}(y))\eta_{\omega} (x)
 \\&+(I\otimes R_{\succ}(b))\tau\Delta_{\prec}(a)\bullet (I\otimes R_{\diamond}(y))\hat{\tau}\eta_{\omega} (x)-
  (I\otimes L_{\prec}(b))\tau\Delta_{\prec}(a)\bullet (I\otimes L_{\diamond}(y))\hat{\tau}\eta_{\omega} (x)
  \\&+(R_{\succ}(b)\otimes I)\tau\Delta_{\succ}(a)\bullet ( R_{\diamond}(y)\otimes I)\hat{\tau}\eta_{\omega} (x)
 +( L_{\prec}(b)\otimes I)\tau\Delta_{\succ}(a)\bullet ( L_{\diamond}(y)\otimes I)\hat{\tau}\eta_{\omega} (x)
 \\&-(R_{\succ}(b)\otimes I)\Delta_{\prec}(a)\bullet ( R_{\diamond}(y)\otimes I)\eta_{\omega} (x)
 +( L_{\prec}(b)\otimes I)\Delta_{\prec}(a)\bullet ( L_{\diamond}(y)\otimes I)\eta_{\omega} (x).
\end{align*}\end{small}
 Combining Eqs.~\eqref{Bg1}-\eqref{Bg5}, we get that
 \begin{small} \begin{align*}&
(\delta-\hat{\tau}\delta)((a\otimes x)\cdot (b\otimes y))-(L_{\cdot}(a\otimes x)\hat{\otimes} I)(\delta
-\hat{\tau}\delta)(b\otimes y)-(I\hat{\otimes} L_{\cdot}(a\otimes x))(\delta-\hat{\tau}\delta)(b\otimes y)
\\&-(I\hat{\otimes} R_{\cdot}(b\otimes y))\delta(a\otimes x)
 -( R_{\cdot}(b\otimes y)\hat{\otimes} I)\hat{\tau}\delta(a\otimes x)
\\=&B(1)+B(2)+B(3)+B(4)+B(5)+B(6),
 \end{align*}\end{small}
 where
\begin{small} \begin{align*}B(1)=&\Big(\Delta(a\succ b)-(I\otimes R_{\succ}(b))\Delta_{\succ}(a)
+(R_{\odot}(b)\otimes I)\tau\Delta_{\succ}(a)+(R_{\odot}(b)\otimes I)\Delta_{\prec}(a)
\\&-(L_{\succ}(a)\otimes I+I\otimes L_{\succ}(a))\Delta(b)
-(I\otimes R_{\succ}(b))\tau\Delta_{\prec}(a)\Big)\bullet \eta_{\omega} (x\diamond y),
\\B(2)=&\Big((I\otimes L_{\prec}(b))\Delta_{\succ}(a)-\Delta(b\prec a)
+(R_{\odot}(b)\otimes I)\Delta_{\prec}(a)-(I\otimes R_{\succ}(b))\tau\Delta_{\prec}(a)
\\&+(I\otimes R_{\prec}(a)-L_{\succ}(a)\otimes I)\Delta(b)
-(L_{\odot}(a)\otimes I)\tau\Delta(b)\Big)\bullet \eta_{\omega} (y\diamond x),
\\B(3)=&\Big((R_{\succ}(b)\otimes I)(\tau\Delta_{\succ}(a)+\Delta_{\prec}(a))
-(I\otimes R_{\odot}(b))\Delta_{\succ}(a)-\tau\Delta(a\succ b)
\\&-(I\otimes R_{\odot}(b))\tau\Delta_{\prec}(a)
+(L_{\succ}(a)\otimes I+I\otimes L_{\succ}(a))\tau\Delta( b)\Big)\bullet \hat{\tau}\eta_{\omega} (x\diamond y),
\\B(4)=&\Big(\tau\Delta( b\prec a)-(I\otimes R_{\odot}(b) )\tau\Delta_{\prec}(a)
-(L_{\prec}(b)\otimes I )\tau\Delta_{\succ}(a)+(R_{\succ}(b)\otimes I )\Delta_{\prec}(a)\\&
+(I\otimes L_{\odot}(a))\Delta(b)+(R_{\prec}(a)\otimes I+I\otimes L_{\succ}(a))\tau\Delta(b)
\Big)\bullet \hat{\tau}\eta_{\omega} (y\diamond x),
\\B(5)=&\Big(
(R_{\prec}(a)\otimes I )\Delta(b)-(I\otimes R_{\succ}(b) )\tau\Delta_{\prec}(a)
\Big)\bullet (R_{\diamond}(x)\otimes I)\eta_{\omega} (y),
\\B(6)=&\Big((R_{\succ}(b)\otimes I )\Delta_{\prec}(a)-(I\otimes R_{\prec}(a) )\tau\Delta(b)
\Big)\bullet \hat{\tau}(R_{\diamond}(x)\otimes I)\eta_{\omega} (y).
\end{align*}\end{small}
Then
B(1)=B(3)=0 follows from Eq.~\eqref{B2}, B(2)=B(4)=0 follows from Eqs.~\eqref{B1}-\eqref{B2}
 and B(5)=B(6)=0 follows from Eq.~\eqref{B3}. It follows that Eq.~\eqref{Pb1} holds.
By Eqs.~\eqref{P1} and \eqref{Pd1}, we get
\begin{small} \begin{align*}&
\delta\Big((a\otimes x)\cdot (b\otimes y)-(b\otimes y)\cdot (a\otimes x)\Big)-\Big((I\otimes (R_{\cdot}-L_{\cdot})(b\otimes y))\Big)\delta(a\otimes x)
\\&-\Big(I \otimes(L_{\cdot}(a\otimes x)-R_{\cdot}(a\otimes x))\Big)\delta(b\otimes y)
-\Big(L_{\cdot}(a\otimes x)\otimes I\Big)\delta(b\otimes y)+\Big(L_{\cdot}(b\otimes y)\otimes I\Big)\delta(a\otimes x)
\\=&\Delta_{\succ}(a\succ b+a\prec b)\bullet \eta_{\omega} (x\diamond y)-\Delta_{\succ}(b\succ a+b\prec a)\bullet \eta_{\omega} (y\diamond x)
\\&-\tau\Delta_{\prec}(a\succ b+a\prec b)\bullet \hat{\tau}\eta_{\omega} (x\diamond y)
+\tau\Delta_{\prec}(b\succ a+b\prec a)\bullet \tau\eta_{\omega} (y\diamond x),
\\&+(I\otimes L_{\circ}(b))\Delta_{\succ}(a)\bullet (I\otimes L_{\diamond}(y))\eta_{\omega} ( x)
-(I\otimes R_{\circ}(b))\Delta_{\succ}(a)\bullet (I\otimes R_{\diamond}(y))\eta_{\omega} ( x)
\\&-(I\otimes L_{\circ}(b))\tau\Delta_{\prec}(a)\bullet (I\otimes L_{\diamond}(y))\hat{\tau}\eta_{\omega} ( x)
+(I\otimes R_{\circ}(b))\tau\Delta_{\prec}(a)\bullet (I\otimes R_{\diamond}(y))\hat{\tau}\eta_{\omega} ( x)
\\&-(I\otimes L_{\circ}(a))\Delta_{\succ}(b)\bullet (I\otimes L_{\diamond}(x))\eta_{\omega} ( y)
+(I\otimes R_{\circ}(a))\Delta_{\succ}(b)\bullet (I\otimes R_{\diamond}(x))\eta_{\omega} ( y)
\\&+(I\otimes L_{\circ}(a))\tau\Delta_{\prec}(b)\bullet (I\otimes L_{\diamond}(x))\hat{\tau}\eta_{\omega} ( y)
+(I\otimes R_{\circ}(a))\tau\Delta_{\prec}(b)\bullet (I\otimes R_{\diamond}(x))\hat{\tau}\eta_{\omega} ( y)
\\&+(L_{\succ}(a)\otimes I)\tau\Delta_{\prec}(b)\bullet ( L_{\diamond}(x)\otimes I)\hat{\tau}\eta_{\omega} ( y)
-(R_{\prec}(a)\otimes I)\tau\Delta_{\prec}(b)\bullet ( R_{\diamond}(x)\otimes I)\hat{\tau}\eta_{\omega} ( y)
\\&-(L_{\succ}(a)\otimes I)\Delta_{\succ}(b)\bullet ( L_{\diamond}(x)\otimes I)\eta_{\omega} ( y)
+(R_{\prec}(a)\otimes I)\Delta_{\succ}(b)\bullet ( R_{\diamond}(x)\otimes I)\eta_{\omega} ( y)
\\&-(L_{\succ}(b)\otimes I)\tau\Delta_{\prec}(a)\bullet ( L_{\diamond}(y)\otimes I)\hat{\tau}\eta_{\omega} ( x)
+(R_{\prec}(b)\otimes I)\tau\Delta_{\prec}(a)\bullet ( R_{\diamond}(y)\otimes I)\hat{\tau}\eta_{\omega} ( x)
\\&+(L_{\succ}(b)\otimes I)\Delta_{\succ}(a)\bullet ( L_{\diamond}(y)\otimes I)\eta_{\omega} ( x)
-(R_{\prec}(b)\otimes I)\Delta_{\succ}(a)\bullet ( R_{\diamond}(y)\otimes I)\eta_{\omega} ( x)\\
=&B(7)+B(8)+B(9)+B(10)+B(11)+B(12)+B(13)+B(14).\end{align*}\end{small}
Combining Eqs.~\eqref{Bg1}-\eqref{Bg5}, we get that
\begin{small} \begin{align*}B(7)=&\Big(\Delta_{\succ}(a\circ b)+(L_{\odot}(b)\otimes I-I\otimes R_{\circ}(b))\Delta_{\succ}(a)+
(L_{\odot}(b)\otimes I-I\otimes R_{\circ}(b))\tau\Delta_{\prec}(a)
\\&-(I\otimes L_{\circ}(a)+L_{\succ}(a)\otimes I)\Delta_{\succ}(b)\Big)\bullet \eta_{\omega} (x\diamond y),
\\B(8)=&\Big((I\otimes L_{\circ}(b)+L_{\odot}(b)\otimes I)\Delta_{\succ}(a)-\Delta_{\succ}(b\circ a)
-(L_{\odot}(a)\otimes I)\tau\Delta_{\prec}(b)
\\&+(I\otimes R_{\circ}(a)-L_{\succ}(a)\otimes I)\Delta_{\succ}(b)-(I\otimes R_{\circ}(b))\tau\Delta_{\prec}(a)
\Big)\bullet \eta_{\omega} (y\diamond x),
\\B(9)=&\Big(\tau\Delta_{\prec}(b\circ a)+
(I\otimes L_{\star}(a))(\Delta_{\succ}(b)-(I\otimes L_{\star}(b))
\tau\Delta_{\prec}(a))+
(I\otimes L_{\circ}(a)-R_{\prec}(a)\otimes I)\tau\Delta_{\prec}(b)\\&-
(L_{\succ}(b)\otimes I)\tau\Delta_{\prec}(a)+(R_{\prec}(b)\otimes I)\Delta_{\succ}(a)
\Big)\bullet \hat{\tau}\eta_{\omega} (y\diamond x),
\\B(10)=&\Big((I\otimes L_{\circ}(a))\tau\Delta_{\prec}(b)-\tau\Delta_{\prec}(a\circ b)-
(I\otimes L_{\star}(b))(\Delta_{\succ}(a)+\tau\Delta_{\prec}(a))
\\&+(L_{\succ}(a)\otimes I)\Delta_{\prec}(b)
+(R_{\prec}(b)\otimes I)(\Delta_{\succ}(a)+\tau\Delta_{\prec}(a))
\Big)\bullet \hat{\tau}\eta_{\omega} (x\diamond y),
\\B(11)=&\Big((R_{\prec}(a)\otimes I)\Delta_{\succ}(b)-(I\otimes R_{\circ}(b))\tau\Delta_{\prec}(a)
\Big)\bullet (R_{\diamond}(x)\otimes I)\eta_{\omega} ( y),
\\B(12)=&\Big((R_{\prec}(b)\otimes I)\Delta_{\succ}(a)-(I\otimes R_{\circ}(a))\tau\Delta_{\prec}(b)
\Big)\bullet \hat{\tau}(R_{\diamond}(x)\otimes I)\eta_{\omega} ( y).
\end{align*}
\end{small}
Using Eqs.~\eqref{B5}-\eqref{B6}, we get B(7)=0. By Eq.~\eqref{B6}, we obtain B(11)=B(12)=0.
And B(8)=0 follows from Eqs.~\eqref{B3} and ~\eqref{B5}
B(9)=0 follows from Eqs.~\eqref{B4} -\eqref{B6}. Thus, Eq.~\eqref{Pb2} holds.
The proof is completed.
 \end{proof}
 
\subsection{Infinite dimensionsal pre-Lie bialgebras from the Leibniz-dendriform Yang-Baxter equation}
 
 Since a solution of the $S$-equation with an invariant symmetric part gives rise to a pre-Lie bialgebra \cite{WBLS},
  it is natural to ask whether such solutions can be derived from the Leibniz-dendriform Yang-Baxter equation.
  We will answer this question in what follows.
  
  We begin by recalling the Leibniz-dendriform Yang-Baxter equation and quasi-triangular Leibniz-dendriform bialgebras.

\begin{defi}\cite{SG} Let $(A,\succ,\prec)$ be a Leibniz-dendriform algebra and
$r\in A\otimes A$. The following equation
\begin{equation} \label{YE1}D(r)=r_{23}\circ r_{13}-r_{12}\odot r_{23}-r_{12}\succ r_{13}=0\end{equation}
is called the {\bf Leibniz-dendriform Yang-Baxter equation} in $(A,\succ,\prec)$ or {\bf LD-YBE} in short.
\end{defi}

\begin{ex} \cite{SG} \label{Zq3}
Let $A$ be a 4-dimensional Leibniz-dendriform algebra with a basis $\{\varepsilon_1, \varepsilon_2,\varepsilon_3, \varepsilon_4\}$, 
where the two binary
operations $\succ,\prec:A \otimes A \longrightarrow A$ are as follows (only non-zero operations are listed):
\begin{align*}&\varepsilon_2\succ \varepsilon_1=\varepsilon_2\succ \varepsilon_2=\varepsilon_1, \ \ \ \ \ \ \ \ \
\varepsilon_3\succ \varepsilon_2=\varepsilon_3\succ \varepsilon_1=-\varepsilon_4, \\&
\varepsilon_2\succ \varepsilon_3=-\varepsilon_3-\varepsilon_4, \ \
\varepsilon_3\prec \varepsilon_1=\varepsilon_4,\ \
 \varepsilon_3\prec \varepsilon_2=\varepsilon_3+ 2\varepsilon_4.
\end{align*}
By direct calculation, $r=\varepsilon_1\otimes \varepsilon_4-\varepsilon_4\otimes \varepsilon_1$ is 
a skew-symmetric solution of the LD-YBE in $(A,\succ,\prec)$.
\end{ex}

\begin{defi} \cite{SG}\label{In1}
 Let $(A,\succ,\prec)$ be a Leibniz-dendriform algebra and $r\in A\otimes A$. Then $r$ is called {\bf invariant} in $A\otimes A$
 if
 \begin{align}&\label{Iv1}
(L_{\odot}(x)\otimes I-I\otimes R_{\circ}(x))r=0,
\\&\label{Iv2}(L_{\star}(x)\otimes I-I\otimes R_{\prec}(x))\tau(r)=0,
\end{align}
 for all $x\in A$, where
$R_{\circ}=R_{\succ}+R_{\prec},~
 L_{\star}=L_{\circ}+R_{\circ},~L_{\odot}=L_{\succ}+R_{\prec}$.
\end{defi}

\begin{thm} \cite{SG}\label{Qy0} Let $(A,\succ,\prec)$ be a Leibniz-dendriform algebra and $r=\sum_{i}a_i\otimes b_i\in A\otimes A$.
 Assume that
$\Delta_{\succ,r},\Delta_{\prec,r}$ are given by 
 the following equations:
\begin{align}&\label{CB1}
\Delta_{\succ}(x)=\Delta_{\succ,r}(x)=(L_{\odot}(x)\otimes I-I\otimes R_{\circ}(x))r,
\\&\label{CB2}\Delta_{\prec}(x)=\Delta_{\prec,r}(x)=(L_{\star}(x)\otimes I-I\otimes R_{\prec}(x))\tau(r),
\end{align}
 for all $x\in A$, where
$\circ=\succ+\prec,~
 L_{\star}=L_{\circ}+R_{\circ},~L_{\odot}=L_{\succ}+R_{\prec}$. 
If $r$ is a solution of the
LD-YBE in $(A,\succ,\prec)$ and $r+\tau(r)$ is invariant.
Then $(A,\succ,\prec,\Delta_{\succ,r},\Delta_{\prec,r})$ is a Leibniz-dendriform bialgebra.
\end{thm}

\begin{defi} \cite{SG} \label{Qt1}
 Let $(A,\succ,\prec)$ be a Leibniz-dendriform algebra and $r\in A\otimes A$. If $r$ is a solution of the LD-YBE in $(A,\succ,\prec)$
 and $r+\tau(r)$ is invariant, then the Leibniz-dendriform
  bialgebra $(A, \succ,\prec,\Delta_{\succ,r},\Delta_{\prec,r})$ induced by $r$ is called a {\bf quasi-triangular} Leibniz-dendriform
 bialgebra. In particular, if $r$ is skew-symmetric, $(A, \succ,\prec,\Delta_{\succ,r},\Delta_{\prec,r})$
is called a {\bf triangular} Leibniz-dendriform bialgebra, where $\Delta_{\succ,r}$ and $\Delta_{\prec,r}$
are given by Eqs.~~(\ref{CB1})-(\ref{CB2}).
\end{defi}

For further details on Leibniz-dendriform bialgebras, we refer the reader to \cite{SG}.

 \begin{defi} \cite{LH,B}
Let ($L=\oplus_{i\in \mathbb{Z}}L_i ,\cdot)$ be a $\mathbb{Z}$-graded pre-Lie algebra.
If $\hat{r}\in L\hat{\otimes} L$ satisfies the {\bf $S$-equation } 
\begin{equation*} S(\hat{r})=\hat{r}_{12}\cdot\hat{r}_{13}-\hat{r}_{12}\cdot \hat{r}_{23}-[\hat{r}_{13}, \hat{r}_{23}]=0.
\end{equation*}
Then $r$ is called a {\bf completed solution of the $S$-equation} in $L$. 
 In the special case where $L = L_0$, it reduces to the solution of the $S$-equation.
\end{defi} 

 \begin{defi} 
 Let ($L=\oplus_{i\in \mathbb{Z}}L_i ,\cdot)$ be a $\mathbb{Z}$-graded pre-Lie algebra
  and $\hat{r}\in L\hat{\otimes} L$. Then $\hat{r}$ is called {\bf invariant} 
 in $L\hat{\otimes }L$
 if
 \begin{align*}
\big(L_{\cdot}(x)\hat{\otimes} I+I\hat{\otimes} (L_{\cdot}(x)-R_{\cdot}(x))\big)\hat{r}=0,~~\forall~x\in L.
\end{align*}
\end{defi}
 
 \begin{thm} \label{Py}
Let ($L=\oplus_{i\in \mathbb{Z}}L_i ,\cdot)$ be a $\mathbb{Z}$-graded pre-Lie algebra. Assume that
 $\hat{r}\in L\hat{\otimes} L$ is a completed solution 
of the $S$-equation whose symmetric part is invariant.
Define a linear map $\delta_{\hat{r}}:L\rightarrow L\hat{\otimes} L$ by
\begin{equation} \label{CB3}
\delta_{\hat{r}}(x)=-\big(L_{\cdot}(x)\hat{\otimes} I+I\hat{\otimes} (L_{\cdot}(x)-R_{\cdot}(x))\big)\hat{r},\quad x\in L.
\end{equation}
 Then $(L, \cdot,\delta_{\hat{r}})$ is a completed pre-Lie bialgebra.
\end{thm}
\begin{proof}
The proof follows by the same argument as Proposition 2.10 in \cite{WBLS}.
\end{proof}
 
 \begin{defi} \label{Qt1}
 Let $(L, \cdot)$ be a $\mathbb{Z}$-graded pre-Lie algebra and $\hat{r}\in L\hat{\otimes} L$. If $\hat{r}$ is a completed solution 
 of the $S$-equation in $(L, \cdot)$
 and $\hat{r}-\hat{\tau}(\hat{r})$ is invariant in $L\hat{\otimes} L$, then the completed pre-Lie
  bialgebra $(L, \cdot,\delta_{\hat{r}})$ induced by $\hat{r}$ is called a {\bf completed quasi-triangular} pre-Lie
 bialgebra, where $\delta_{\hat{r}}$ is given by Eq.~\eqref{CB3}.
 In particular, if $\hat{r}$ is symmetric, the same bialgebra is called a completed triangular pre-Lie
  is called a {\bf completed triangular} pre-Lie bialgebra.
\end{defi}

 We establish a connection between solutions of the LD-YBE in a 
 Leibniz-dendriform algebra and completed solutions of the $S$-equation in the induced pre-Lie algebra.

 \begin{thm} \label{Tb4}
Let $(A, \succ,\prec)$ be a Leibniz-dendriform algebra and $(B=\oplus_{i\in\mathbb{Z}}B_{i}, \diamond,
\omega)$ be a quadratic $\mathbb{Z}$-graded Zinbiel algebra. Assume that $(A\otimes B, \cdot)$ is the
induced $\mathbb{Z}$-graded pre-Lie algebra from $(A, \succ,\prec)$ 
by $(B=\oplus_{i\in\mathbb{Z}}B_{i}, \diamond,\omega)$.
 Suppose that $r=\sum_{i}u_{i}
\otimes v_{i}\in A\otimes A$ is a solution of the LD-YBE in $(A, \succ,\prec)$ with $r+\tau(r)$ invariant in $A\otimes A$. Then
\begin{align}\label{Cr1} 
\hat{r}=\sum_{i}\sum_{j\in\Omega}(u_{i}\otimes e_{j})\otimes(v_{i}\otimes f_{j})
\in(A\otimes B)\,\hat{\otimes}\,(A\otimes B)  
\end{align}
is a completed solution of the S-equation in $(A\otimes B, \cdot)$ with $\hat{r}-\hat{\tau} \hat{r}$
 invariant in $(A\otimes B)\hat{\otimes}(A\otimes B)$. In particular, if 
 $r=\sum_{i}u_{i}
\otimes v_{i}\in A\otimes A$ is a skew-symmetric solution of the LD-YBE in $(A, \succ,\prec)$, then
$\hat{r}=\sum_{i}\sum_{j\in\Omega}(u_{i}\otimes e_{j})\otimes(v_{i}\otimes f_{j})
\in(A\otimes B)\,\hat{\otimes}\,(A\otimes B) $
is a symmetric completed solution of the $S$-equation in $(A\otimes B, \cdot)$,
where $\{e_{j}\}_{j\in\Omega}$ is a homogeneous basis of $B=\oplus_{i\in\mathbb{Z}}B_{i}$
and $\{f_{j}\}_{j\in\Omega}$ is its homogeneous dual basis with respect to $\omega$. 
\end{thm}

\begin{proof}
Since $\omega$ is invariant, for given $e_{s}, e_{t}, e_{p}\in B,~(s, t, p\in\Omega)$, we have 
\begin{small}
\begin{align}&\label{Bg6}\sum_{p,q}e_{p}\otimes e_{q}\diamond f_{p} \otimes f_{q}=
-\sum_{p,q}e_{p}\otimes e_{q}\otimes f_{q}\diamond f_{p}=\sum_{p,q}\big(e_{p}\diamond e_{q}\otimes f_{p}\otimes f_{q}
		 + e_{q}\diamond e_{p}\otimes f_{p}\otimes f_{q}\big),\\
		&\label{Bg7}\sum_{p,q} e_{p} \diamond e_{q} \otimes f_{p} \otimes f_{q} = -\sum_{p,q} e_{p} \otimes f_{p}\diamond e_{q} \otimes f_{q}.
\\&\label{Bg8}\sum_{p,q}e_{q}\diamond e_{p}\otimes f_{p}\otimes f_{q}
		=\sum_{p,q}e_{p}\otimes e_{q}\otimes f_{p}\diamond f_{q}, 
\\&\label{Bg9}\sum_{j} e_{p}\diamond f_{j} \otimes e_{j}=-\sum_{j} e_{p}\diamond e_{j} \otimes f_{j}, 
\\&\label{Bg10}\sum_{j}e_{j}\otimes f_{j}\diamond e_{p} =-\sum_{j} e_{j}\diamond e_{p} \otimes f_{j}.
\\&\label{Bg11}\sum_{j}e_{j}\otimes f_{j}\diamond e_{p} = \sum_{j} f_{j} \diamond e_{p} \otimes e_{j}=-\sum_{j} f_{j} \otimes e_{j}\diamond e_{p} , 
\\&\label{Bg12}\sum_{j}f_{j}\otimes e_{p}\diamond e_{j}=-\sum_{j}e_{j}\otimes e_{p}\diamond f_{j} =\sum_{j}e_{j}\otimes
		 f_{j}\diamond e_{p} +\sum_{j} e_{p}\diamond f_{j} \otimes e_{j}.
	\end{align}
\end{small}
Further details on the proof of the above equations are available in \cite{HL}.
By Eq.~\eqref{Pd1}, we obtain
\begin{small} \begin{align}
&\label{Yq1}(L_{\cdot}(a\otimes e_k)\hat{\otimes} I+I\hat{\otimes}(L_{\cdot}-R_{\cdot})(a\otimes e_k) )(\hat{r}-\hat{\tau}\hat{r})
\\=&\sum_{i}\sum_{j} (a\otimes e_k)\cdot(u_i\otimes e_j)\otimes v_i\otimes f_j-(a\otimes e_k)\cdot(v_i\otimes f_j)\otimes u_i\otimes e_j
\nonumber\\&+u_i\otimes e_j\otimes[a\otimes e_k,v_i\otimes f_j]-v_i\otimes f_j\otimes[a\otimes e_k,u_i\otimes e_j]
\nonumber\\=&(L_{\succ}(a)\otimes I)r\bullet(e_k\diamond e_j\otimes f_j)-(R_{\prec}(a)\otimes I)r\bullet(e_j\diamond e_k\otimes f_j)
\nonumber\\&-(L_{\succ}(a)\otimes I)\tau(r)\bullet(e_k\diamond f_j\otimes e_j)+(R_{\prec}(a)\otimes I)\tau(r)\bullet(f_j\diamond e_k\otimes e_j)
\nonumber\\&+(I\otimes L_{\succ}(a))r\bullet(e_j\otimes e_k\diamond f_j)-(I\otimes R_{\prec}(a))r\bullet(e_j\otimes f_j\diamond e_k)
\nonumber\\&-(I\otimes R_{\succ}(a))r\bullet(e_j\otimes f_j\diamond e_k)+(I\otimes L_{\prec}(a))r\bullet(e_j\otimes e_k\diamond f_j)
\nonumber\\&-(I\otimes L_{\succ}(a))\tau(r)\bullet(f_j\otimes e_k\diamond e_j)+(I\otimes R_{\prec}(a))\tau(r)\bullet(f_j\otimes e_j\diamond e_k)
\nonumber\\&+(I\otimes R_{\succ}(a))\tau(r)\bullet(f_j\otimes e_j\diamond e_k)-(I\otimes L_{\prec}(a))\tau(r)\bullet(f_j\otimes e_k\diamond e_j)
\nonumber\\=&\sum_{j}(I\otimes L_{\star}(a)-R_{\prec}(a)\otimes I)(r+\tau(r))\bullet \big(e_j\otimes(e_k\diamond f_j)\big)
\nonumber\\&+(L_{\odot}(a)\otimes I-I\otimes R_{\circ}(a))(r+\tau(r))\bullet \big(e_k\diamond e_j\otimes f_j\big)\nonumber.
\end{align}
and
\begin{align}
&\label{Yq2}S(\hat{r})=\hat{r}_{12}\cdot\hat{r}_{13}-[\hat{r}_{13},\hat{r}_{23}]
-\hat{r}_{12}\cdot\hat{r}_{23}\\
=&\sum_{i,j}\sum_{p,q}\big(u_{i}\succ u_{j}\otimes v_{i}\otimes v_{j}\big)
\bullet\big(e_{p}\diamond e_{q}\otimes f_{p}\otimes f_{q}\big)
-\big(u_{j}\prec u_{i}\otimes v_{i}\otimes v_{j}\big)
\bullet\big(e_{q}\diamond e_{p}\otimes f_{p}\otimes f_{q}\big)\nonumber\\
&-\big(u_{i}\otimes u_{j}\otimes(v_{i}\succ v_{j}\big)\bullet
\big(e_{p}\otimes e_{q}\otimes f_{p}\diamond f_{q}\big)
+\big(u_{i}\otimes u_{j}\otimes v_{j}\prec v_{i}\big)\bullet
\big(e_{p}\otimes e_{q}\otimes f_{q}\diamond f_{p}\big)
\nonumber\\&+\big(u_{i}\otimes u_{j}\otimes v_{j}\succ v_{i}\big)
\bullet\big(e_{p}\otimes e_{q}\otimes f_{q}\diamond f_{p}\big)
-\big(u_{i}\otimes u_{j}\otimes v_{i}\prec v_{j}\big)
\bullet\big(e_{p}\otimes e_{q}\otimes f_{p}\diamond f_{q}\big)
\nonumber\\&-\big(u_{i}\otimes v_{i}\succ u_{j}\otimes v_{j}\big)
\bullet\big(e_{p}\otimes f_{p}\diamond e_{q}\otimes  f_{q}\big)
+\big(u_{i}\otimes  u_{j}\prec v_{i}\otimes v_{j}\big)
\bullet\big(e_{p}\otimes e_{q}\diamond f_{p}\otimes  f_{q}\big)
\nonumber\\=&\sum_{p,q}\big(r_{12}\succ r_{13}\big)
\bullet\big(e_{p}\diamond e_{q}\otimes f_{p}\otimes f_{q}\big)
-\big(r_{13}\prec r_{12}\big)
\bullet\big(e_{q}\diamond e_{p}\otimes f_{p}\otimes f_{q}\big)\nonumber\\
&-\big(r_{13}\succ r_{23}\big)\bullet
\big(e_{p}\otimes e_{q}\otimes f_{p}\diamond f_{q}\big)
+\big(r_{23}\prec r_{13}\big)\bullet
\big(e_{p}\otimes e_{q}\otimes f_{q}\diamond f_{p}\big)
\nonumber\\&+\big(r_{23}\succ r_{13}\big)
\bullet\big(e_{p}\otimes e_{q}\otimes f_{q}\diamond f_{p}\big)
-\big(r_{13}\prec r_{23}\big)
\bullet\big(e_{p}\otimes e_{q}\otimes f_{p}\diamond f_{q}\big)
\nonumber\\&-\big(r_{12}\succ r_{23}\big)
\bullet\big(e_{p}\otimes f_{p}\diamond e_{q}\otimes  f_{q}\big)
+\big(r_{23}\prec r_{12}\big)
\bullet\big(e_{p}\otimes e_{q}\diamond f_{p}\otimes  f_{q}\big)\nonumber.
\end{align}\end{small}
In view of $r+\tau(r)$ being invariant in $A\otimes A$ and Eqs.~\eqref{Bg9}-\eqref{Bg12}, we get that
\begin{small} \begin{align*}&
(L_{\cdot}(a\otimes e_k)\hat{\otimes} I+I\hat{\otimes}(L_{\cdot}-R_{\cdot})(a\otimes e_k) )(\hat{r}-\hat{\tau}\hat{r})
\\=&\sum_{j}(I\otimes L_{\star}(a)-R_{\prec}(a)\otimes I)(r+\tau(r))\bullet \big(e_j\otimes(e_k\diamond f_j)\big)
+(L_{\odot}(a)\otimes I-I\otimes R_{\circ}(a))(r+\tau(r))\bullet \big(e_k\diamond e_j\otimes f_j\big)\\=&0,
\end{align*}\end{small}
which means that $\hat{r}-\hat{\tau}\hat{r}$
 is invariant in $(A\otimes B)\hat{\otimes}(A\otimes B)$.
Assume that $r$ is a solution of the LD-YBE $D(r)=0$ and $r+\tau(r)$ is invariant in $A\otimes A$.
By Theorem 4.4 \cite{SG}, we obtain that $D(r)+D_{1}(r)=0$, where
$D_{1}(r)=r_{12}\odot r_{13}+r_{13}\circ r_{23}+r_{12}\succ r_{23}.$
Combining Eqs.~\eqref{Bg6}-\eqref{Bg8}, we have
\begin{align*}
&S(\hat{r})=\hat{r}_{12}\cdot\hat{r}_{13}-[\hat{r}_{13},\hat{r}_{23}]
-\hat{r}_{12}\cdot\hat{r}_{23}
\\=&\sum_{p,q}D(r)\bullet\big (e_{p}\diamond e_{q}\otimes f_{p}\otimes f_{q}\big)
+\big (D(r)+D_{1}(r)\big )\bullet \big(e_{q}\diamond e_{p}\otimes f_{p}\otimes f_{q}\big)\\=&0.
\end{align*}
Hence, $\hat{r} $ is a completed solution of the S-equation $S(\hat{r})=0 $ in $(A\otimes B,\cdot)$.

Note that for all $e_p, e_q \in B$ with $p, q \in \Omega$,
\begin{small}\begin{align*}
\hat{\omega}\Big(\sum_{j\in\Omega}e_{j}\otimes f_{j},\; e_{p}\otimes e_{q}\Big)
=-\omega(e_{q}, e_{p})=\omega(e_{p}, e_{q})
=-\hat{\omega}\Big(\sum_{j\in\Omega}f_{j}\otimes e_{j},\; e_{p}\otimes e_{q}\Big).
\end{align*}\end{small}
It follows that
$\sum_{j\in\Omega}e_{j}\otimes f_{j}
=-\sum_{j\in\Omega}f_{j}\otimes e_{j}$. Thus, if $r$ is skew-symmetric, then $\hat{r}$
is symmetric. 
 \end{proof}

\begin{thm}
Let $(A, \succ,\prec,\Delta_{\succ},\Delta_{\prec})$ be a 
Leibniz-dendriform bialgebra and $(B=\oplus_{i\in\mathbb{Z}}B_{i}, \diamond,
\omega)$ a quadratic $\mathbb{Z}$-graded Zinbiel algebra. Assume that $(A\otimes B, \cdot,\delta)$ is
the induced completed pre-Lie bialgebra from $(A, \succ,\prec,\Delta_{\succ},\Delta_{\prec})$ by
$(B, \diamond, \omega)$. If $(A, \succ,\prec,\Delta_{\succ,r},\Delta_{\prec,r})$ is quasi-triangular,
 then $(A\otimes B, \cdot,\delta_{\hat{r}})$ is also quasi-triangular.
 In particular, $(A\otimes B, \cdot,\delta_{\hat{r}})$ is also triangular if 
 $(A, \succ,$ \quad  \quad \quad $\prec,\Delta_{\succ,r},\Delta_{\prec,r})$ is triangular.
\end{thm}

\begin{proof} Let $(A, \succ,\prec,\Delta_{\succ,r},\Delta_{\prec,r})$ be a
quasi-triangular Leibniz-dendriform bialgebra. Assume that
 $r=\sum_{i}u_{i}\otimes v_{i}$ is a solution of the LD-YBE $D(r)=0$ in $(A, \succ,\prec)$ 
and $r+\tau(r)$ is invariant in $A\otimes A$.
By Theorem \ref{Tb4}, we know that $\hat{r}$ is a completed solution
 of the $S$-equation $S(\hat{r})=0$ in $(A\otimes B, \cdot)$ and $\hat{r}-\hat{\tau}\hat{r}$
 is invariant in $(A\otimes B)\hat{\otimes} (A\otimes B)$.
It follows that $(A\otimes B, \cdot, \delta_{\hat{r}})$ is a quasi-triangular
completed pre-Lie bialgebra, where
$\delta_{\hat{r}}$ is given by Eq.~\eqref{CB3}. The proof is finished.
\end{proof}
 
\begin{ex} By Theorem \ref{Tb4}, a symmetric completed solution of the $S$-equation in $(A\otimes V, \cdot)$ can be obtained 
  from Example \ref{Zq3} and Example \ref{Zq2}, the details are omitted here.
\end{ex}
 
 In view of Theorems \ref{Tb3} and \ref{Tb4}, from a solution of the LD-YBE with invariant skew-symmetric part, 
 we can obtain two completed pre-Lie bialgebras via two methods. 
 It is natural to consider whether these two completed pre-Lie bialgebras coincide.
 
 \begin{thm} \label{Ty1}
 Under the same assumption as in Theorem \ref{Tb4}.
	Let $\Delta_{\succ,r},\Delta_{\prec,r}:A\rightarrow A\otimes A$ be the linear maps defined by Eqs.~\eqref{CB1}-\eqref{CB2}, 
and $\delta: A\otimes B \rightarrow (A\otimes B ) \hat{\otimes} (A\otimes B )$ 
be the linear map defined by Eq.~\eqref{P1}. Then $(A\otimes B, \cdot, \delta)$ is a pre-Lie bialgebra by Theorem~\ref{Tb3}.
It coincides with the quasi-triangular pre-Lie bialgebra $(A\otimes B, \cdot, \delta_{\hat{r}})$
with $\delta_{\hat{r}}: A\otimes B \rightarrow (A\otimes B ) \hat{\otimes} (A\otimes B )$ defined by Eq.~\eqref{CB3}, where
$\hat{r}$ is given by Eq.~\eqref{Cr1}. That is, 
 we have the following commutative diagram:
 \begin{equation*}
        \xymatrix@C=3cm{
            \txt{ solutions of the LD-YBE \\ with invariant skew-symmetric parts} \ar[r]^-{Theorem~\ref{Qy0}}
             \ar[d]^-{Theorem ~\ref{Tb4}} & \txt{Leibniz-dendriform bialgebras} \ar[d]^-{Theorem~\ref{Tb3}}\\
            \txt{solutions of the S-equation \\ with invariant symmetric parts} 
            \ar[r]^-{Theorem~\ref{Py}} & \txt{pre-Lie bialgebras}}
    \end{equation*}
\end{thm}
\begin{proof} Following \cite{HL}, we have
	\begin{align*}
		&\eta_{\omega}(x) = -\sum_{j}x \diamond e_{j} \otimes f_{j},\\
		&\sum_{j} e_{j} \otimes f_{j}\diamond x =-\sum_{j} e_{j}\diamond x \otimes f_{j}, \\
		& \hat{\tau} \eta_{\omega}(x) = \sum_{j} e_{j} \otimes x \diamond f_{j} = \sum_{j} x \diamond e_{j} \otimes f_{j} +
		\sum_{j} e_{j} \diamond x \otimes f_{j}.
	\end{align*}
Combining Eqs.~\eqref{P1}, \eqref{CB1}-\eqref{CB2}, \eqref{CB3} and \eqref{Cr1}, we obtain
\begin{align*}
		&\delta_{\hat{r}}(a \otimes x)=-\Big(L_{\cdot}(a\otimes x)\otimes I+I\otimes (L_{\cdot}-R_{\cdot})(a\otimes x)\Big)
\sum_{i,j}(u_{i}\otimes e_{j} \otimes v_{i}\otimes f_{j})
\\=&-(a\otimes x)\cdot(u_{i}\otimes e_{j})\otimes (v_{i}\otimes f_{j})-(u_{i}\otimes e_{j}) \otimes
(a\otimes x)\cdot(v_{i}\otimes f_{j})
+(u_{i}\otimes e_{j}) \otimes
(v_{i}\otimes f_{j})\cdot(a\otimes x)
\\=&\sum_{i,j}\big((u_i\prec a) \otimes v_i\big)\bullet\big((e_{j}\diamond x)\otimes f_{j}\big)
-\big((a\succ u_i)\otimes v_i\big)\bullet\big((x\diamond e_{j})\otimes f_{j}\big)
\\&-\big(u_{i}\otimes (a\succ v_i)\big)\bullet\big(e_{j}\otimes (x\diamond f_{j})\big)
+\big(u_{i}\otimes ( v_i\prec a)\big)\bullet\big(e_{j}\otimes (f_{j}\diamond x)\big)
-\big(u_{i}\otimes ( a\prec v_i)\big)\bullet\big(e_{j}\otimes (x\diamond f_{j})\big)
\\=&\sum_{j}\big(R_{\prec}(a)\otimes I-I\otimes L_{\star}(a)\big)r\bullet\big((e_{j}\diamond x)\otimes f_{j}\big)
-\big(L_{\succ}(a)\otimes I+I\otimes L_{\circ}(a)\big)r\bullet\big((x\diamond e_{j})\otimes f_{j}\big)
\\=&\big(R_{\prec}(a)\otimes I-I\otimes L_{\star}(a)\big)r\bullet\big(\eta_{\omega}(x)+\hat{\tau}\eta_{\omega}(x))
+\big(L_{\succ}(a)\otimes I+I\otimes L_{\circ}(a)\big)r\bullet\eta_{\omega}(x)
\\=&\big(L_{\odot}(a)\otimes I-I\otimes R_{\circ}(a)\big)r\bullet\eta_{\omega}(x)
+\big(R_{\prec}(a)\otimes I-I\otimes L_{\star}(a)\big)r\bullet\hat{\tau}\eta_{\omega}(x)
\\=&\Delta_{\succ,r}(a)\bullet\eta_{\omega}(x)
-\tau \Delta_{\prec,r}(a) \bullet\hat{\tau}\eta_{\omega}(x)
\end{align*}
This finishes the proof.
\end{proof}
 
 \section{ Infinite dimensionsal pre-Lie bialgebras induced from Zinbiel-dendriform bialgebras and quadratic Leibniz algebras}
 
In this section, we equip the tensor product of a Zinbiel-dendriform bialgebra 
and a quadratic Leibniz algebra with a pre-Lie bialgebra structure, 
and prove that quasi-triangularity (resp. triangularity) is inherited from the Zinbiel-dendriform bialgebras.
 Furthermore, the completed pre-Lie bialgebra obtained from the original structure via 
 a special quadratic Leibniz algebra recovers the original; 
 this construction is referred to as the affinization of Zinbiel-dendriform bialgebras.

\subsection{ Zinbiel-dendriform algebras and bialgebras}

We begin by recalling the bialgebra theory for Zinbiel-dendriform algebras, 
which are precisely commutative quadri-algebras \cite{AL}.
 As a special case of the latter, this theory is readily obtained from the framework in \cite{NB}.

\begin{defi} \cite{AL}
	A {\bf Zinbiel-dendriform algebra} is a triple $(A, \rhd,\lhd)$, where $A$ is a vector space 
and $\rhd,\lhd: A \otimes A \rightarrow A$ are binary operations satisfying
	\begin{align*}
	x\rhd(y\rhd z)&=(x\ast y+y\ast x)\rhd z,\\
	x\rhd(z\lhd y)&=z\lhd(x\ast y)=(x\rhd z)\lhd y+(z\lhd x)\lhd y, \;\; \forall x, y, z\in A.
	\end{align*}
where $x\ast y=x\rhd y+x\lhd y$.
  $(A,\ast)$ is a Zinbiel algebra, which is called the \textbf{sub-adjacent Zinbiel algebra} of $(A,\succ,\prec)$
  and $(A,\succ,\prec)$ is called the \textbf{compatible Zinbiel-dendriform algebra} on $(A,\ast)$.
\end{defi}

\begin{rmk}
Quadri-algebras were introduced in \cite{AL}. A particular case, namely commutative quadri-algebras, 
will be referred to as Zinbiel-dendriform algebras.
\end{rmk}

\begin{defi} \cite{NB} A Zinbiel-dendriform coalgebra is a triple $(D,\beta_{\rhd},\beta_{\lhd})$, where
$D$ is a vector space and $\beta_{\rhd},\beta_{\lhd}:D\longrightarrow D\otimes D$ are linear maps such that
the following conditions hold:
\begin{align}&\label{Zd1}
( I\otimes \beta_{\rhd})\beta_{\rhd}=(\tau \otimes I)(I\otimes\beta_{\rhd})\beta_{\rhd}
=((\beta_{\rhd}+\tau\beta_{\rhd}+\beta_{\lhd}
+\tau\beta_{\lhd})\otimes I)\beta_{\rhd},\\&
\label{Zd2}(I\otimes (\beta_{\rhd}+\beta_{\lhd}))\beta_{\lhd}=
(\beta_{\lhd}\otimes I)\beta_{\lhd}+(\tau \otimes I)(\beta_{\rhd}\otimes I)\beta_{\lhd}
=(\tau \otimes I)(I\otimes \beta_{\lhd})\beta_{\rhd}.\end{align}
\end{defi}

\begin{defi} \cite{NB} \label{DB2} A Zinbiel-dendriform bialgebra is a quintuple $(D,\rhd,\lhd,\beta_{\rhd},\beta_{\lhd})$
such that $(D,\rhd,\lhd)$ is a Zinbiel-dendriform algebra, $(D,\beta_{\rhd},\beta_{\lhd})$ is
a Zinbiel-dendriform coalgebra and the following compatible conditions hold:
\begin{small}
\begin{align}&\label{Zd3}
\beta_{\lhd}(x\ast y+y\ast x)=(L_{\ast}(y)\otimes I )\beta_{\lhd}(x)+(I\otimes L_{\ominus}(x))\beta_{\lhd}(y),
\\&\label{Zd4}\beta(x\ominus y)=(L_{\rhd}(x)\otimes I)\beta(y)
+(I\otimes R_{\ominus}(y))\beta_{\rhd}(x)=(I\otimes L_{\ominus}(x))\beta( y)+(L_{\lhd}(y)\otimes I)\beta_{\lhd}(x),
\\&\label{Zd5}\beta_{\ominus}(x\ast y)=(L_{\rhd}(x)\otimes I)\beta_{\ominus}(y)+(I\otimes R_{\ast}(y))\beta_{\rhd}(x),
\\&\label{Zd6}(L_{\lhd}(y)\otimes I)\tau\beta_{\ominus}(x)+(I\otimes R_{\lhd}(x))\beta(y)=(\beta+\tau\beta)(y\lhd x),
\\&\label{Zd7}(I\otimes L_{\ast}(x)-L_{\rhd}(x)\otimes I)\beta_{\ominus}(y)
=(I\otimes R_{\ast}(y))\beta_{\rhd}(x)-(R_{\lhd}(y)\otimes I)\tau\beta_{\lhd}(x),
\end{align}
\end{small}
where $x\ast y=x \rhd y+x\lhd y,~x\ominus y=x\rhd y+y\lhd x,~\beta=\beta_{\lhd}+\beta_{\rhd}, ~\beta_{\ominus}=\tau\beta_{\lhd}+\beta_{\rhd}
$ and $R_{\ast}=R_{\rhd}+R_{\lhd},
~L_{\ast}=L_{\rhd}+L_{\lhd},~L_{\ominus}=R_{\lhd}+L_{\rhd},~~R_{\ominus}=L_{\lhd}+R_{\rhd}$.
\end{defi}
By Eqs.~\eqref{Zd3}, \eqref{Zd4} and \eqref{Zd6}, we obtain that
\begin{align}&\label{Zd8}
(L_{\ast}(y)\otimes I )\beta_{\lhd}(x)+(I\otimes L_{\ominus}(x))\beta_{\lhd}(y)=
(L_{\ast}(x)\otimes I )\beta_{\lhd}(y)+(I\otimes L_{\ominus}(y))\beta_{\lhd}(x),
\\&\label{Zd9}(I\otimes L_{\lhd}(y))\beta_{\ominus}(x)+(R_{\lhd}(x)\otimes I)\tau\beta(y)
=(L_{\lhd}(y)\otimes I)\tau\beta_{\ominus}(x)+(I\otimes R_{\lhd}(x))\beta(y),
\\&\label{Zd10}\beta_{\rhd}(x\ast y+y\ast x)=(L_{\rhd}(y)\otimes I)\beta_{\rhd}(x)+(I\otimes (L_{\ast}+R_{\ast})(x)\beta_{\rhd}(y).
\end{align}

\begin{defi} \cite{NB} Let $(P,\rhd,\lhd)$ be a Zinbiel-dendriform 
algebra and let $\mathfrak{B}$ be a symmetric bilinear
form on $P$. If $\mathfrak{B}$ satisfies
\begin{align}&\label{Cq1}\mathfrak{B}(x\rhd y,z)=\mathfrak{B}(y,x\ast z+z\ast x),
\\&\label{Cq2} \mathfrak{B}(x\lhd y,z)=-\mathfrak{B}(y,z\rhd x+x\lhd z),
\end{align} then $\mathfrak{B}$ is
called {\bf invariant} on $P$, where $x\ast z=x\rhd z+x\lhd z$.
\end{defi}

\begin{defi} \cite{NB}
Let $(P,\rhd,\lhd)$ be a Zinbiel-dendriform algebra. Suppose that there is a Zinbiel-dendriform algebra structure
on $P^*$. If there is a Zinbiel-dendriform algebra structure on the direct sum
of the underlying vector space of $P$ and $P^*$ such that $P$ and
$P^*$ are Zinbiel-dendriform subalgebras and the bilinear form
$\mathfrak{B}_S$ on $P\oplus P^*$ given by 
\begin{equation}\label{Cq3}\mathfrak{B}_S(x+\zeta, y+\eta) =\langle x,\theta\rangle+\langle \zeta,y\rangle,  
  \ \ ~\forall~x, y\in P,~\zeta,\theta\in P^{*}.\end{equation} is
invariant, then $(P\oplus P^*,P,P^*,\mathfrak{B}_S)$ is called a
{\bf (standard) Manin triple } of Zinbiel-dendriform algebras associated to the
nondegenerate invariant bilinear form $\frak B_S$.
\end{defi}

\begin{thm} \cite{NB}
 Let $(P,\rhd_P,\lhd_P,\beta_{\rhd},\beta_{\lhd})$ be a Zinbiel-dendriform algebra. Assume that
there is a Zinbiel-dendriform algebra $(P^{*}, \rhd_{P^*},\lhd_{P^*})$ which is induced from
a Zinbiel-dendriform coalgebra $(P, \beta_{\rhd},\beta_{\lhd})$. Then
$(P,\rhd_P,\lhd_P,\beta_{\rhd},\beta_{\lhd})$ is a Zinbiel-dendriform bialgebra
if and only if
$(P\oplus P^*,P,P^*,\mathfrak{B}_S)$ is a standard Manin triple
of Zinbiel-dendriform algebras associated to the nondegenerate invariant
bilinear form $\frak B_S$ given by Eq.~\eqref{Cq3}. 
\end{thm}

\begin{defi}  \cite{NB}
Let $(M,\rhd,\lhd)$ be a Zinbiel-dendriform algebra and $r\in M\otimes M$. The following equation 
\begin{equation}Z(r)=r_{12}\ominus r_{23}-r_{13}\ast r_{23}+
r_{13}\lhd r_{12}=0.
\end{equation}
is called the Zinbiel-dendriform
Yang-Baxter equation in $(M,\rhd,\lhd)$ or the ZD-YBE in short.
\end{defi}

\begin{ex} \label{Lq3}
Let $P$ be a $3$-dimensional vector space with basis $\{u,v,w\}$. 
Define the binary operations $\rhd$ and $\lhd$ on $P$ by the following nonzero products:
\[
u\rhd u = v,\qquad u\lhd u = w,
\]
and set all other products to be zero. By direct computations, $(P,\rhd,\lhd)$ is a Zinbiel-dendriform algebra.
Moreover, $r=v\otimes w-w\otimes v$ is a skew-symmetric solution of the ZD-YBE in $(P,\rhd,\lhd)$.
\end{ex}

\begin{defi} \label{In1}
 Let
$(M,\rhd,\lhd)$ be a Zinbiel-dendriform algebra and
let $r\in M\otimes M$.
 Then $r$ is called {\bf invariant} in $M\otimes M$
 if \begin{align}&\label{Iv3}
(I\otimes L_{\ominus}(m)- L_{\ast}(m)\otimes I)r=0,
\\&\label{Iv4}(L_{\rhd}(m)\otimes I-I\otimes (L_{\ast}+R_{\ast})(m) )r=0.
\end{align}
\end{defi}

\begin{pro}\label{Qy} Let $(M,\rhd,\lhd)$ be a Zinbiel-dendriform algebra and
 and $r=\sum_{i}a_i\otimes b_i\in M\otimes M$. 
Assume that $r+\tau(r)$ is invariant in $M\otimes M$. If $r$ is a solution of the
ZD-YBE in $(M,\rhd,\lhd)$, then $r$ is also a solution of the following equation:
\begin{align*}Z_{1}(r)=r_{13}\ast r_{23}+r_{23}\ast r_{13}+r_{12}\rhd r_{13}-r_{12}\rhd r_{23}=0.\end{align*}
\end{pro}

\begin{proof}
Since $r+\tau(r)$ is invariant in $M\otimes M$, it follows that
\begin{align*}\sum_{i}(R_{\ominus}(u_i)\otimes I+I\otimes L_{\lhd}(u_i))(r+\tau(r))\otimes v_i=0.\end{align*}
Note that 
\begin{align*}
&Z_{1}(r)=-Z(r)-\sigma_{12}Z(r)-(r_{12}+r_{21})\ominus r_{13}-r_{23}\lhd (r_{12}+r_{21})
\\=&-Z(r)-\sigma_{12}Z(r)-\sum_{i}(R_{\ominus}(u_i)\otimes I+I\otimes L_{\lhd}(u_i))(r+\tau(r))\otimes v_i
\\=&0.
\end{align*}
The proof is completed.
\end{proof}

Following the same procedure, we obtain the following result.

\begin{thm}\label{Qy1} Let $(M,\rhd,\lhd)$ be a Zinbiel-dendriform algebra and
 and $r=\sum_{i}a_i\otimes b_i\in M\otimes M$. Assume that $r$ is a solution of the
ZD-YBE in $(M,\rhd,\lhd)$ and $r+\tau(r)$ is invariant.
Then $(M,\rhd,\lhd,\beta_{\rhd,r},\beta_{\lhd,r})$ is a Zinbiel-dendriform bialgebra, where
$\beta_{\rhd,r},\beta_{\lhd,r}$ are given by
\begin{align}&\label{CB4}
\beta_{\lhd,r}(m)=(I\otimes L_{\ominus}(m)- L_{\ast}(m)\otimes I)r,
\\&\label{CB5}\beta_{\rhd,r}(m)=(L_{\rhd}(m)\otimes I-I\otimes (L_{\ast}+R_{\ast})(m) )r.
\end{align}
\end{thm}

\begin{defi} Let $(M,\rhd,\lhd)$ be a Zinbiel-dendriform algebra and
 and $r\in M\otimes M$. Assume that $r$ is a solution of the
ZD-YBE in $(M,\rhd,\lhd)$ and $r+\tau(r)$ is invariant.
Then the Zinbiel-dendriform
  bialgebra $(M,\rhd,\lhd,\beta_{\rhd,r},\beta_{\lhd,r})$ induced by $r$ is called a {\bf quasi-triangular} Zinbiel-dendriform
 bialgebra. In particular, if $r$ is skew-symmetric, $(M,\rhd,\lhd,\beta_{\rhd,r},\beta_{\lhd,r})$
is called a {\bf triangular} Zinbiel-dendriform bialgebra, where $\beta_{\rhd,r}$ and $\beta_{\lhd,r}$
are given by Eqs.~(\ref{CB4})-(\ref{CB5}) respectively.
\end{defi}

\subsection{Pre-Lie (co)algebras induced from Zinbiel-dendriform (co)algebras and Leibniz (co)algebras}
In this subsection, we endow the tensor product of a Zinbiel-dendriform (co)algebra 
and a Leibniz (co)algebra with a pre-Lie (co)algebra structure.

A {\bf completed Leibniz coalgebra} \cite{HL} is a pair $(Q, \vartheta)$, where $Q=\oplus_{i\in\mathbb{Z}}Q_{i}$ 
is a $\mathbb{Z}$-graded vector space with all $Q_i$ finite-dimensional 
and $\vartheta: Q\rightarrow Q\hat{\otimes}Q$ is a linear map satisfying
		\begin{equation} \label{Lc1}
			(I\hat{\otimes}\vartheta)\vartheta=(\vartheta\hat{\otimes}I)\vartheta +(\hat{\tau}\hat{\otimes}I)(I\hat{\otimes}\vartheta)\vartheta. 
		\end{equation}
		When $Q = Q_{0}$, it is simply called a {\bf Leibniz coalgebra}.    
    By Eq.~\eqref{Lc1}, we have
{\small
	\begin{flalign}&\label{Lc2}(\vartheta\hat{\otimes}I)\vartheta=-(\hat{\tau}\hat{\otimes}I)(\vartheta\hat{\otimes}I)\vartheta,
\\&\label{Lc3}(I\hat{\otimes}\vartheta)\hat{\tau}\vartheta
=-(I\hat{\otimes}\hat{\tau})(I\hat{\otimes}\vartheta)\hat{\tau}\vartheta
=(\hat{\tau}\hat{\otimes} I)(I\hat{\otimes} \hat{\tau})(I\hat{\otimes}\vartheta)\vartheta
-(\hat{\tau}\hat{\otimes} I)(\vartheta\hat{\otimes}I)\hat{\tau}\vartheta,
\\&\label{Lc4}
(I\hat{\otimes}  \hat{\tau})(\vartheta\hat{\otimes}I)\vartheta
=-(\hat{\tau}\hat{\otimes} I)(I\hat{\otimes} \hat{\tau})(I\hat{\otimes}\vartheta)\hat{\tau}\vartheta=
(I\hat{\otimes}\hat{\tau})(I\hat{\otimes}\vartheta)\vartheta-(\vartheta\hat{\otimes}I)\hat{\tau}\vartheta.\end{flalign}}

\begin{ex}\label{Lq0}
Let \(W=\operatorname{span}\{\varepsilon_1,\varepsilon_2,\varepsilon_3,\varepsilon_4\}\). Define the non-zero products
\begin{equation*}\varepsilon_2\circ \varepsilon_1=\varepsilon_1, 
~ \varepsilon_2\circ \varepsilon_2=\varepsilon_1, 
~\varepsilon_2\circ \varepsilon_3=-\varepsilon_3-\varepsilon_4,~
\varepsilon_3\circ \varepsilon_1=\varepsilon_4, 
~ \varepsilon_3\circ \varepsilon_2=\varepsilon_3+2\varepsilon_4,\end{equation*}
and all others zero. Then $(W,\circ)$ is a Leibniz algebra. 
Assume that
\[Q=W\otimes \mathbb K[t,t^{-1}]=\bigoplus_{n\in\mathbb Z}(W\otimes t^n)\]
equipped with \[(x\otimes t^m)\circ(y\otimes t^n)=(x\circ y)\otimes t^{m+n},\]
then $(Q,\circ)$
is a $\mathbb Z$-graded Leibniz algebra. Moreover, define a
comultiplication $\vartheta:Q\longrightarrow Q\hat{\otimes}Q$
 by~(denote $\varepsilon_kt^m=\varepsilon_k\otimes t^m,~k=1,2,3,4)$,
\[\begin{aligned}\vartheta(\varepsilon_1t^n)&
=\sum_{i+j=n}\varepsilon_1t^i\otimes \varepsilon_4t^j-\varepsilon_4t^i\otimes \varepsilon_1t^j
,\\[6pt]\vartheta(\varepsilon_2t^n)
&=\sum_{i+j=n}\varepsilon_1t^i\otimes \varepsilon_3t^j+2\varepsilon_1t^i\otimes \varepsilon_4t^j-\varepsilon_4t^i\otimes \varepsilon_1t^j,
\\[6pt]\vartheta(\varepsilon_3t^n)&=\sum_{i+j=n}\varepsilon_4t^i\otimes \varepsilon_3t^j+\varepsilon_4t^i\otimes \varepsilon_4t^j,
\\[6pt]\vartheta(\varepsilon_4t^n)&=0.
\end{aligned}\]
A direct computation shows that $( Q,\vartheta)$ is a completed Leibniz coalgebra.
\end{ex}

\begin{pro}\label{Tb5}
Let $(P, \rhd,\lhd)$ be a Zinbiel-dendriform algebra and $(Q=\oplus_{i\in\mathbb{Z}}Q_{i}, \circ)$ be a $\mathbb{Z}$-graded Leibniz algebra.
Define a binary operation $\cdot$ on $P\otimes Q$ by
\begin{small}\begin{align} \label{Pd2}
(a\otimes x)\cdot(b\otimes y)
=a\rhd b\otimes x\circ y-b\lhd a\otimes y\circ x, \end{align}\end{small}
for any $a,b\in P$ and $x,y\in Q$. Then
$(P\otimes Q, \cdot)$ is a pre-Lie algebra, which is called the {\bf pre-Lie algebra induced 
from $(P, \rhd,\lhd)$ by a $(Q, \circ)$}. Moreover,
 if $(Q=\oplus_{i\in\mathbb{Z}}Q_{i}, \circ)$ is the $\mathbb{Z}$-graded Leibniz algebra given in
Example \ref{Lq0}, then $(P\otimes Q, \cdot)$ is a $\mathbb{Z}$-graded pre-Lie algebra
 if and only if $(P, \rhd,\lhd)$ is a Zinbiel-dendriform algebra.
\end{pro}

\begin{proof}
To prove that $(P\otimes Q, \cdot)$ is a pre-Lie algebra, we only need to check that
$(a\otimes x,b\otimes y,c\otimes z)=(b\otimes y,a\otimes x,c\otimes z)$ for all $a, b, c\in A$, $x,y,z\in B$.
In fact, using Eq.~\eqref{Pd1} we obtain
\begin{small}\begin{align*}
&(a\otimes x,b\otimes y,c\otimes z)
\\=&(a\rhd  b)\rhd c\otimes  (x\circ y)\circ z-c\lhd(a\rhd b)\otimes z\circ (x\circ y)
\\&- (b\lhd a)\rhd c\otimes  (y\circ x)\circ z+c\lhd (b\lhd a)\otimes z\circ (y\circ x)
-a\rhd  (b\rhd c)\otimes  x\circ (y\circ z)\\&+(b\rhd c)\lhd a\otimes (y\circ z)\circ x 
+ a\rhd(c\lhd b)\otimes x \circ(z\circ y) -(c\lhd b)\lhd a\otimes (z\circ y)\circ x,
\end{align*}\end{small}
and
\begin{small}\begin{align*}
&(b\otimes y,a\otimes x,c\otimes z)
\\=&(b\rhd  a)\rhd c\otimes  (y\circ x)\circ z-c\lhd(b\rhd a)\otimes z\circ (y\circ x)
\\&- (a\lhd b)\rhd c\otimes  (x\circ y)\circ z+c\lhd (a\lhd b)\otimes z\circ (x\circ y)
-b\rhd  (a\rhd c)\otimes  y\circ (x\circ z)\\&+(a\rhd c)\lhd b\otimes (x\circ z)\circ y 
+ b\rhd (c\lhd a)\otimes y \circ(z\circ x) -(c\lhd a)\lhd b\otimes (z\circ x)\circ y.
\end{align*}\end{small}
Combining Eqs.~\eqref{Z1} and ~\eqref{Ld1}-\eqref{Ld3}, we get
\begin{small}\begin{align*}
&(a\otimes x,b\otimes y,c\otimes z)-(b\otimes y,a\otimes x,c\otimes z)
\\=&[(a\rhd  b+b\rhd  a+a\lhd  b+b\lhd a)\rhd c-a\rhd ( b\rhd c)]\otimes  x\circ (y\circ z)
+[b\rhd ( a\rhd c)-a\rhd ( b\rhd c)]\otimes  y\circ(x\circ z)
\\&+[a\rhd ( c\lhd b)-c\lhd ( a\rhd b+a\lhd b)]\otimes  x\circ (z\circ y)
+[( a\rhd c+c\lhd a)\lhd b-c\lhd ( a\rhd b+a\lhd b)]\otimes  (z\circ x)\circ y
\\&+[c\lhd ( b\rhd a+b\lhd a)-( b\rhd c+c\lhd b)\lhd a]\otimes  (z\circ y)\circ x
+[c\lhd ( b\rhd a+b\lhd a)-b\rhd (c\lhd a)]\otimes  y\circ(z\circ x)
\\&=0.
\end{align*}\end{small}
Conversely,
 suppose that $(P\otimes Q, \cdot)$ is a $\mathbb{Z}$-graded pre-Lie algebra
 with $(Q=\oplus_{i\in\mathbb{Z}}Q_{i}, \circ)$
being the $\mathbb{Z}$-graded Leibniz algebra given in Example \ref{Lq0}. 
Comparing the coefficients of
 $t^{m+n+l}\varepsilon_4$ 
 in the following expansions respectively
 \begin{align*}&(a\otimes \varepsilon_2t^{m},b\otimes \varepsilon_1t^{n},c\otimes \varepsilon_3t^{l})
 =(b\otimes \varepsilon_1t^{n},a\otimes \varepsilon_2t^{m},c\otimes \varepsilon_3t^{l}),
 \\&(a\otimes \varepsilon_3t^{m},b\otimes \varepsilon_1t^{n},c\otimes \varepsilon_2t^{l})
 =(b\otimes \varepsilon_1t^{n},a\otimes \varepsilon_3t^{m},c\otimes \varepsilon_2t^{l}),\\&
 (a\otimes \varepsilon_2t^{m},b\otimes \varepsilon_3t^{n},c\otimes \varepsilon_1t^{l})
 =(b\otimes \varepsilon_3t^{n},a\otimes \varepsilon_2t^{m},c\otimes \varepsilon_1t^{l}),\end{align*}
 we obtain that
 \begin{align*}
 &c\lhd (a\rhd b+a\lhd b)=(a\rhd c+c\lhd a)\lhd b,\\&
 (a\rhd b+a\lhd b+b\rhd a+b\lhd a)\rhd c=b\rhd (a\rhd c),\\&
 (a\rhd c+c\lhd a)\lhd b=a\rhd (c\lhd b),
 \end{align*}
 that is, $(P, \rhd,\lhd)$ is a Zinbiel-dendriform algebra.
The proof is completed.
\end{proof}
  
 \begin{pro} \label{Tb6}
Let $(P, \beta_{\rhd},\beta_{\lhd})$ be a Zinbiel-dendriform coalgebra and
$(Q,\vartheta)$ be a completed Leibniz coalgebra.
 Define a linear map $\delta: P\otimes Q\rightarrow(P\otimes Q)\otimes(P\otimes Q)$ by
\begin{small}\begin{align}
\delta(a\otimes x)
&=\beta_{\rhd}(a)\bullet \vartheta(x)-\tau\beta_{\lhd}(a)\bullet \tau\vartheta(x)
 \label{P2} \\
&=\sum_{(a)}\sum_{i,j,\alpha}\Big((a_{1}\otimes x_{1,i,\alpha})
\otimes(a_{2}\otimes x_{2,j,\alpha})-(a_{(2)}\otimes x_{2,j,\alpha})
\otimes(a_{(1)}\otimes x_{1,i,\alpha})\Big),  \nonumber
\end{align}\end{small}
for any $a,b\in P$ and $x,y\in Q$, where
 $\beta_{\rhd}(a)=\sum_{(a)}a_{1}\otimes a_{2},~\beta_{\lhd}(a)=\sum_{(a)}a_{(1)}\otimes a_{(2)}$
 and  $\vartheta(x)=\sum_{i,j,\alpha}x_{1,i,\alpha}\otimes x_{2,j,\alpha}$.
 Then $(P\otimes Q, \delta)$ is a pre-Lie coalgebra, which is called the \textbf{pre-Lie coalgebra}
 induced from $(P,\beta_{\rhd},\beta_{\lhd})$ by $(Q, \vartheta)$.
Moreover, if $(Q=\oplus_{i\in\mathbb{Z}}Q_{i}, \vartheta)$ is the completed Leibniz coalgebra given in
Example \ref{Lq0}, then $(P\otimes Q, \delta)$ is a completed pre-Lie coalgebra 
if and only if $(P, \beta_{\rhd},\beta_{\lhd})$ is a Zinbiel-dendriform coalgebra.
\end{pro}

\begin{proof}
Replacing $(A, \Delta_{\succ},\Delta_{\prec})$ and $(B,\eta)$ by
 $(P, \beta_{\rhd},\beta_{\lhd})$ and $(Q,\vartheta)$ respectively
 in the proof of Proposition \ref{Tb2}, combining Eqs.~\eqref{Lc1}-\eqref{Lc4}, we obtain that
{\small
	\begin{flalign*}
		&(\delta\hat{\otimes} I)\delta(a\otimes x)-(I\hat{\otimes} \delta)\delta(a\otimes x)
-(\hat{\tau}\hat{\otimes} I)(\delta\hat{\otimes} I)\delta(a\otimes x)
+(\hat{\tau}\hat{\otimes} I)(I\hat{\otimes} \delta)\delta(a\otimes x)
\\=&P(1)+P(2)+P(3)+P(4)+P(5)+P(6),
\end{flalign*}}
where
{\small
	\begin{flalign*}
		&P(1)=[(\tau\otimes I)(\beta_{\lhd}\otimes I)\tau\beta_{\lhd}(a)+(\tau\otimes I)(\beta_{\rhd}\otimes I)\tau\beta_{\lhd}(a)
-(\tau\otimes I)(I\otimes \tau)(I\otimes \beta_{\lhd})\beta_{\rhd}(a)]
\bullet (\hat{\tau} \hat{\otimes} I)(\vartheta\hat{\otimes} I)\hat{\tau}\vartheta (x),\\&
P(2)=[(I\otimes \tau)(I\otimes \beta_{\lhd})\tau\beta_{\lhd}(a)
-(\tau\otimes I)(I\otimes \tau)(I\otimes \beta_{\lhd})\beta_{\rhd}(a)
+(\tau\otimes I)(I\otimes \tau)(\beta_{\rhd}\otimes I)\beta_{\lhd}(a)]\bullet(I\hat{\otimes} \vartheta)\hat{\tau}\vartheta (x),
\\&P(3)=[(I\otimes \tau)(I\otimes \beta_{\lhd})\beta_{\rhd}(a)-
(I\otimes \tau)(\beta_{\rhd}\otimes I)\beta_{\lhd}(a)-
(\tau\otimes I)(I\otimes \tau)(I\otimes \beta_{\lhd})\tau\beta_{\lhd}(a)]
\bullet(I\hat{\otimes} \hat{\tau})(\vartheta\hat{\otimes}I )\vartheta (x),
\\&P(4)=[(I\otimes \tau)(I\otimes \beta_{\lhd})\beta_{\rhd}(a)
-(\beta_{\lhd}\otimes I)\tau\beta_{\lhd}(a)-(\beta_{\rhd}\otimes I)\tau \beta_{\lhd}(a)]
\bullet (\vartheta \hat{\otimes} I)\hat{\tau}\vartheta (x),
\\&P(5)=[(\beta_{\rhd}+\beta_{\lhd})\otimes I)\beta_{\rhd}(a)-(\tau \otimes I)
((\beta_{\rhd}+\beta_{\lhd})\otimes I)\beta_{\rhd}(a)]\bullet (\vartheta\hat{\otimes} I)\vartheta (x),
\\&P(6)=[(I\otimes \beta_{\rhd})\beta_{\rhd}(a)-(\tau\otimes I)(I\otimes \beta_{\rhd})\beta_{\rhd}(a)]
\bullet(\hat{\tau}\hat{\otimes} I)(I\hat{\otimes} \vartheta)\vartheta (x),
\end{flalign*}}
Then, 
P(1)=P(2)=P(3)=P(4)=0 follows from Eq.~\eqref{Zd2}, P(5)= P(6)=0 follows from Eq.~\eqref{Zd1}.

On the other hand, suppose that $(P\otimes Q, \delta)$ is a completed pre-Lie coalgebra with 
$(Q=\oplus_{i\in\mathbb{Z}}Q_{i}, \vartheta)$ being the completed Leibniz coalgebra given in
Example \ref{Lq0}.
Comparing the coefficients of
 $\varepsilon_3t^m\otimes \varepsilon_4t^l\otimes \varepsilon_4t^{n-m-l}$ and $\varepsilon_4t^m\otimes \varepsilon_4t^l\otimes \varepsilon_3t^{n-m-l}$
  respectively in the expansion
 \begin{align*}(\delta\hat{\otimes} I)\delta(a\otimes \varepsilon_3t^n)-(I\hat{\otimes} \delta)\delta(a\otimes \varepsilon_3t^n)
 -(\hat{\tau}\hat{\otimes} I)(\delta\hat{\otimes} I)\delta(a\otimes \varepsilon_3t^n)
+(\hat{\tau}\hat{\otimes} I)(I\hat{\otimes} \delta)\delta(a\otimes \varepsilon_3t^n)=0,\end{align*}
which implies that 
\begin{align*}&(I\otimes (\beta_{\rhd}+\beta_{\lhd}))\beta_{\lhd}(a)
=(\tau \otimes I)(I\otimes \beta_{\lhd})\beta_{\rhd}(a),\\&
( I\otimes \beta_{\rhd})\beta_{\rhd}=(\tau \otimes I)(I\otimes\beta_{\rhd})\beta_{\rhd}.\end{align*}
Comparing the coefficients of
 $\varepsilon_3t^m\otimes \varepsilon_4t^l\otimes \varepsilon_1t^{n-m-l}$ and
  $\varepsilon_1t^m\otimes \varepsilon_4t^l\otimes \varepsilon_3t^{n-m-l}$
  respectively in the expansion
 \begin{align*}(\delta\hat{\otimes} I)\delta(a\otimes \varepsilon_2t^n)-(I\hat{\otimes} \delta)\delta(a\otimes \varepsilon_2t^n)
 -(\hat{\tau}\hat{\otimes} I)(\delta\hat{\otimes} I)\delta(a\otimes \varepsilon_2t^n)
+(\hat{\tau}\hat{\otimes} I)(I\hat{\otimes} \delta)\delta(a\otimes \varepsilon_2t^n)=0,\end{align*}
which implies that 
\begin{align*}&(\beta_{\lhd}\otimes I)\beta_{\lhd}+(\tau \otimes I)(\beta_{\rhd}\otimes I)\beta_{\lhd}
=(\tau \otimes I)(I\otimes \beta_{\lhd})\beta_{\rhd},\\&
( I\otimes \beta_{\rhd})\beta_{\rhd}=((\beta_{\rhd}+\tau\beta_{\rhd}+\beta_{\lhd}.\end{align*}
Thus, Eqs.~\eqref{Zd1}-\eqref{Zd2} hold, that is, $(P, \beta_{\rhd},\beta_{\lhd})$ is a Zinbiel-dendriform coalgebra.

The proof is completed.
\end{proof}

\begin{rmk} 
The results of Propositions~\ref{Tb5} and~\ref{Tb6} remain valid if 
we replace Leibniz (co)algebras with Leibniz (co)dialgebras. 
However, since the latter are more involved and the former are precisely the commutative counterparts,
 we restrict our treatment to the (co)algebra case for simplicity.
\end{rmk}

\subsection{Pre-Lie bialgebras induced from Zinbiel-dendriform bialgebras and quadratic Leibniz algebras}

Having obtained pre-Lie (co)algebras on the tensor product of Zinbiel-dendriform 
(co)algebras and Leibniz (co)algebras, we are now led to consider how pre-Lie bialgebras can be derived from these structures.

\begin{pro} \label{Sp3}
Let $(P, \rhd,\lhd,\mathfrak{B})$ be a quadratic Zinbiel-dendriform algebra, $(Q, \circ)$ be 
a Leibniz algebra with a non-degenerate bilinear form $\varpi$, and $(P\otimes Q, \cdot)$ be the induced pre-Lie algebra from
 $(P, \rhd,\lhd,\mathfrak{B})$ and $(Q, \circ)$. Then $(P\otimes Q, \cdot, \omega)$ is a quadratic pre-Lie algebra with $\omega$ defined by
\begin{align*}
\omega(a\otimes x,b\otimes y)=\mathfrak{B}(a,b)\varpi(x, y),\;\;\;a, b\in P,\;\;x, y\in Q,
\end{align*}
if and only if $\varpi$ is skew-symmetric and the following equality holds for all $a,b,c\in P$ and $x,y,z\in Q$:
\begin{align*}
\mathfrak{B}( b,a\ast c)\Big(\varpi(x\circ y,z)+\varpi(y,x\circ z)\Big)
+\mathfrak{B}(b,c\ast a)\Big(\varpi(x\diamond y, z)+\varpi(y\circ x,z)-\varpi(y,z\circ x)\Big)=0.
\end{align*}
\end{pro}
\begin{proof} The skew-symmetry of $\varpi$, together with Eqs.~\eqref{Cq1}-\eqref{Cq2}, 
implies that for all $a, b,c\in P$ and $x, y,z\in Q$,
\begin{align*}
	&\omega((a\otimes x)\cdot (b\otimes y),c\otimes z)+\omega(b\otimes y,[a\otimes x,c\otimes z ])
\\=&\mathfrak{B}(a\rhd b,c)\varpi(x\circ y,z)-\mathfrak{B}(b\lhd a,c)\varpi(y\circ x,z)
	+\mathfrak{B}(b,a\rhd c)\varpi(y,x\circ z)\\&-\mathfrak{B}(b,c\lhd a)\varpi(y,z\circ x)
	-\mathfrak{B}(b,c\rhd a)\varpi(y,z\circ x)+\mathfrak{B}(b,a\lhd c)\varpi(y,x\circ z)
\\=&\mathfrak{B}( b,a\ast c+c\ast a)\varpi(x\circ y,z)+\mathfrak{B}(b,c\ast a)\varpi(y\circ x,z)
\\&+\mathfrak{B}(b,a\ast c)\varpi(y,x\circ z)-\mathfrak{B}(b,c\ast a)\varpi(y,z\circ x)
\\=&\mathfrak{B}( b,a\ast c)\Big(\varpi(x\circ y,z)+\varpi(y,x\circ z)\Big)
+\mathfrak{B}(b,c\ast a)\Big(\varpi(x\circ y, z)+\varpi(y\circ x,z)-\varpi(y,z\circ x)\Big),
\end{align*}
where $a\ast b=a\rhd b+a\lhd b$.
This finishes the proof.
\end{proof}
A bilinear form $\varpi$ on a Leibniz algebra $(Q, \circ)$ is called {\bf invariant} if 
\begin{equation*}
	\varpi(b_{1} \circ b_{2},\; b_{3}) 
= \varpi(b_{1},\; b_{2} \circ b_{3} + b_{3} \circ b_{2}), \;\; \forall ~b_{1}, b_{2}, b_{3}\in Q.
\end{equation*}
A quadratic Leibniz algebra \cite{TS} is a Leibniz algebra $(Q, \circ)$ endowed
with a nondegenerate skew-symmetric and invariant bilinear form $\varpi$. 
If $(Q, \circ, \varpi)$ is a quadratic Leibniz algebra, then
 \begin{equation*}\varpi(b_{1} \circ b_{2}, \; b_{3}) = -\varpi(b_{2}, \; b_{1} \circ b_{3}),~~
\forall~b_{1}, b_{2}, b_{3} \in Q.\end{equation*}
 
Proposition \ref{Sp3} implies that if the non-degenerate skew-symmetric bilinear form $\varpi$ satisfies the identity
 $\varpi(z\circ x,y)=\varpi(z,x\circ y)+\varpi(z,y\circ x)$, 
 which is equivalent to $(Q, \circ, \varpi)$ is a quadratic Leibniz algebra,
 then $(P\otimes Q, \cdot, \omega)$ is a quadratic pre-Lie algebra.
 Accordingly, we investigate the construction of (completed) pre-Lie bialgebras 
from Zinbiel-dendriform bialgebras and quadratic Leibniz algebras.
For the infinite-dimensional setting, however, we must additionally consider quadratic 
$\mathbb{Z}$-graded Leibniz algebras.

\begin{defi}\cite{HL}
A {\bf quadratic $\mathbb{Z}$-graded Leibniz algebra}, denoted by 
$(Q = \bigoplus_{i \in \mathbb{Z}} Q_i, \circ, \varpi)$, 
is a $\mathbb{Z}$-graded Leibniz algebra equipped with a nondegenerate skew-symmetric,
invariant and graded bilinear form $\varpi$. In the special case where $Q = Q_0$, 
it reduces to a quadratic Leibniz algebra.\end{defi}

Let $(Q, \circ,\varpi)$ be a quadratic Leibniz
algebra. Define a linear map \(\vartheta_{\varpi}:Q\to Q\otimes Q\) by
\[\hat{\varpi}(\vartheta_{\varpi}(b_1), b_2\otimes b_3) = -\omega(b_1, b_2\circ b_3)\]
for all $b_1,b_2,b_3\in Q$. Then $(Q,\vartheta_{\varpi})$ is a Leibniz coalgebra \cite{HL}.

Assume that $(Q=\oplus_{i\in\mathbb{Z}}Q_{i}, \circ,\varpi)$ is a quadratic $\mathbb{Z}$-graded Leibniz algebra.
From \cite{HL}, recall that
 \begin{small}
	\begin{align}&\label{Bd1}\vartheta_{\varpi}(x\circ y)=\Big(I\otimes R_{\circ}(y)-\hat{\tau}(I\otimes R_{\circ}(y))\Big)\vartheta_{\varpi}(x)+
\Big(\hat{\tau}(I\otimes R_{\circ}(x))-I\otimes R_{\circ}(x)\Big)\vartheta_{\varpi}(y)
-\hat{\tau}\vartheta_{\varpi}(y\circ x),
\\&\label{Bd2}(I\otimes L_{\circ}(y))\vartheta_{\varpi}(x)=-(I\otimes R_{\circ}(y))\vartheta_{\varpi}(x)+\hat{\tau}\Big((I\otimes R_{\circ}(y))\vartheta_{\varpi}(x)
-(I\otimes R_{\circ}(x))\vartheta_{\varpi}(y)-\vartheta_{\varpi}(x\circ y)\Big),
\\&\label{Bd3}(R_{\circ}(x)\otimes I)\vartheta_{\varpi}(y)=(R_{\circ}(y)\otimes I)\vartheta_{\varpi}(x)=0,\\&
\label{Bd4}(L_{\circ}(x)\otimes I)\vartheta_{\varpi}(y)=(I\otimes R_{\circ}(y))\vartheta_{\varpi}(x).	\end{align} 
\end{small}

\begin{ex}\label{Lq1}
Let 
\[ Q=W\otimes \mathbb K[t,t^{-1}]=\bigoplus_{n\in\mathbb Z}(W\otimes t^n)\]
be the infinite-dimensional $\mathbb Z$-graded Leibniz algebra given in Example \ref{Lq0}.
Define a bilinear form \(\varpi\) on $W$ by
\[
\nu(\varepsilon_1,\varepsilon_3)=1,\qquad \nu(\varepsilon_2,\varepsilon_4)=1,
\]
and extend skew-symmetrically. Then $(W,\circ,\nu)$ is a quadratic Leibniz algebra.
Moreover, give a bilinear form \(\varpi\) on $Q$ by
 \[\varpi(x\otimes t^m,y\otimes t^n)=\nu(x,y)\delta_{m+n,0}.\]
 Then $(Q,\circ,\varpi)$
is a quadratic $\mathbb Z$-graded Leibniz algebra. The Leibniz coalgebra from Example \ref{Lq0}
is precisely the one whose comultiplication is induced by \(\varpi\). Denote this 
 Leibniz coalgebra by $(Q,\vartheta_{\varpi})$.
\end{ex}

We are ready to give the affinization of Zinbiel-dendriform bialgebras.

\begin{thm}\label{Tb7}
Let $(P, \rhd,\lhd,\beta_{\rhd},\beta_{\lhd})$ be a Zinbie-dendriform bialgebra
and $(Q=\oplus_{i\in\mathbb{Z}}Q_{i}, \varpi)$ be a quadratic $\mathbb{Z}$-graded Leibniz algebra.
Define a binary operation $\cdot$ on $P\otimes Q$ by Eq.~\eqref{Pd2}, and
a linear map $\delta: P\otimes Q\rightarrow(P\otimes Q
\,\hat{\otimes}\,(P\otimes Q)$ by Eq.~\eqref{P2} with $\vartheta=\vartheta_{\varpi}$.
Then $(P\otimes Q,\cdot,\delta)$ is a completed pre-Lie bialgebra.
Moreover, if $(Q=\oplus_{i\in\mathbb{Z}}Q_{i}, \circ, \varpi)$ is the quadratic $\mathbb{Z}$-graded
Leibniz algebra given in Example \ref{Lq1}, then $(P\otimes Q,\cdot,\delta)$ is a completed pre-Lie bialgebra
if and only if $(P, \rhd,\lhd,\beta_{\rhd},\beta_{\lhd})$ is a Zinbie-dendriform bialgebra.
\end{thm} 

 \begin{proof} Replacing $(A, \Delta_{\succ},\Delta_{\prec})$ and $(B,\eta)$ by
 $(P, \beta_{\rhd},\beta_{\lhd})$ and $(Q,\vartheta)$ respectively
 in the proof of Theorem \ref{Tb3}, and combining Eqs.~\eqref{Bd1}-\eqref{Bd4}, we have 
\begin{small} \begin{align*}&
(\delta-\hat{\tau}\delta)\Big((a\otimes x)\cdot (b\otimes y)\Big)-\Big(L_{\cdot}(a\otimes x)\hat{\otimes} I\Big)(\delta
-\hat{\tau}\delta)(b\otimes y)-\Big(I\hat{\otimes} L_{\cdot}(a\otimes x)\Big)(\delta-\hat{\tau}\delta)(b\otimes y)
\\&-\Big(I\hat{\otimes} R_{\cdot}(b\otimes y)\Big)\delta(a\otimes x)
 -\Big( R_{\cdot}(b\otimes y)\hat{\otimes} I\Big)\hat{\tau}\delta(a\otimes x),
\\=&Q(1)+Q(2)+Q(3)+Q(4)+Q(5)+Q(6),
\end{align*}\end{small}
and
\begin{small}\begin{align*}&
\delta\Big((a\otimes x)\cdot (b\otimes y)-(b\otimes y)\cdot (a\otimes x)\Big)-\Big(I\otimes (R_{\cdot}(b\otimes y)
-L_{\cdot}(b\otimes y))\Big)\delta(a\otimes x)
\\&-\Big(I \otimes(L_{\cdot}-R_{\cdot})(a\otimes x))\Big)\delta(b\otimes y)
-\Big(L_{\cdot}(a\otimes x)\otimes I\Big)\delta(b\otimes y)+\Big(L_{\circ}(b\otimes y)\otimes I\Big)\delta(a\otimes x)
\\
=&Q(7)+Q(8)+Q(9)+Q(10)+Q(11)+Q(12),\end{align*}\end{small}
 where
\begin{small} \begin{align*}Q(1)=&\Big(
(L_{\rhd}(a)\otimes I)\tau\beta(b)-I\otimes L_{\lhd}(b))\beta_{\rhd}(a)-\tau\beta(a\rhd b)+\beta(b\lhd a)\Big)
\bullet \hat{\tau}\vartheta_{\varpi} (x\circ y),
\\Q(2)=&\Big((I\otimes L_{\rhd}(a))\beta(b)-( L_{\lhd}(b)\otimes I)\tau\beta_{\rhd}(a)
-\beta(a\rhd b)+\tau\beta(b\lhd a)\Big)\bullet \hat{\tau}\vartheta_{\varpi} (y\circ x),
\\Q(3)=&\Big(\beta(a\rhd b+b\lhd a)-(L_{\rhd}(a)\otimes I)\beta(b)-
 (I\otimes \big(R_{\rhd}(b)+L_{\lhd}(b))\big)\beta_{\rhd}(a)\Big)\bullet (I\otimes R_{\circ}(y))\vartheta_{\varpi}(x),
\\Q(4)=&\Big((I\otimes (L_{\rhd}+R_{\lhd})(a) )\beta(b)+(L_{\lhd})(b)\otimes I)\beta_{\lhd}(a)
 -\beta(a\rhd b+b\lhd a)\Big)\bullet (I\otimes R_{\circ}(x))\vartheta_{\varpi}(y),
\\Q(5)=&\Big(
(I\otimes L_{\rhd}(a)-L_{\rhd}(a)\otimes I)\tau\beta(b)+(I\otimes L_{\rhd}(a))\beta(b)
+(I\otimes L_{\lhd}(b))\beta_{\rhd}(a)\\&+
(R_{\rhd}(b)\otimes I)\tau\beta_{\rhd}(a)
-\beta(a\rhd b+b\lhd a)\Big)\bullet \hat{\tau}(I\otimes R_{\circ}(y))\vartheta_{\varpi} (x),
\\Q(6)=&\Big((L_{\lhd})(b)\otimes  I)\tau\beta_{\rhd}(a)+\beta(a\rhd b+b\lhd a)
 -(R_{\lhd}(a)\otimes I)\tau\beta(b) \\&-(I\otimes L_{\rhd}(a) )\beta(b)-(I\otimes L_{\lhd}(b) )\big(\beta_{\rhd}(a)+\tau\beta_{\lhd}(a)\big)
\Big)\bullet \hat{\tau}(I\otimes R_{\circ}(x))\vartheta_{\varpi} (y),
\\Q(7)=&\Big(\beta_{\rhd}(b\ast a)-\tau\beta_{\lhd}(a\ast b)+
(L_{\rhd}(a)\otimes I)\tau\beta_{\lhd}(b)-(I\otimes L_{\ast}(b))\beta_{\rhd}(a)\Big)
\bullet \hat{\tau}\vartheta_{\varpi} (x\circ y),
\\Q(8)=&\Big((I\otimes L_{\ast}(a))\beta_{\rhd}(b)-\beta_{\rhd}(a\ast b)+\tau\beta_{\lhd}(b\ast a)-
(L_{\rhd}(b)\otimes I)\tau\beta_{\lhd}(a)\Big)\bullet \hat{\tau}\vartheta_{\varpi} (y\circ x),
\\Q(9)=&\Big(\beta_{\rhd}(a\ast b+b\ast a)-\big(I\otimes (L_{\ast}+R_{\ast})(b)\big)\beta_{\rhd}(a)-
 (L_{\rhd}(a)\otimes I)\beta_{\rhd}(b)
 \Big)\bullet (I\otimes R_{\circ}(y))\vartheta_{\varpi}(x),
\\Q(10)=&\Big((I\otimes (R_{\ast}+L_{\ast})(a) )\beta_{\rhd}(b)+(L_{\rhd}(b)\otimes I)\beta_{\rhd}(a)
 -\beta_{\rhd}(a\ast b+b\ast a)\Big)\bullet (I\otimes R_{\circ}(x))\vartheta_{\varpi}(y),
\\Q(11)=&\Big(\beta_{\rhd}(a\ast b+b\ast a)-(I\otimes L_{\ast}(b))(\beta_{\rhd}(a)+\tau\beta_{\lhd}(a))
-(I\otimes L_{\ast}(a))\beta_{\rhd}(b)\\&-(R_{\lhd}(a)\otimes I)\tau\beta_{\lhd}(b)+
(L_{\rhd}(b)\otimes I)\tau\beta_{\lhd}(a)
\Big)\bullet \hat{\tau}(I\otimes R_{\circ}(x))\vartheta_{\varpi} (y),
\\Q(12)=&\Big((I\otimes L_{\ast}(b))\beta_{\rhd}(a)+(I\otimes L_{\ast}(b))
 (\beta_{\rhd}(b)+\tau\beta_{\lhd}(b))-(L_{\rhd}(a)\otimes I)\tau\beta_{\lhd}(b)
 \\&+ (R_{\lhd}(b)\otimes  I)\tau\beta_{\lhd}(a)-\beta_{\rhd}(a\ast b+b\ast a)
\Big)\bullet \hat{\tau}(I\otimes R_{\circ}(y))\vartheta_{\varpi} (x).
\end{align*}
\end{small}
Then
Q(1)=Q(2)=0 follows from Eqs.~\eqref{Zd4} and \eqref{Zd6}, Q(3)=Q(4)=0 follows from Eq.~\eqref{Zd4},
Q(5)=Q(6)=0 follows from Eqs.~\eqref{Zd4} and \eqref{Zd9}. Likewise,
Q(7)=Q(8)=0 is implied by Eqs.~\eqref{Zd3}, \eqref{Zd5} and \eqref{Zd7}, while
Q(9)=Q(10)=0 follows from Eq.~\eqref{Zd10} and
Q(11)=Q(12)=0 follows from Eqs.~\eqref{Zd3}, \eqref{Zd5} and \eqref{Zd10}.
Hence, Eqs.~\eqref{Pb1}-\eqref{Pb2} hold.  

Conversely, assume that $(P\otimes Q,\cdot,\delta)$ is a completed pre-Lie bialgebra
with $(Q=\oplus_{i\in\mathbb{Z}}Q_{i}, \circ, \varpi)$ being the quadratic $\mathbb{Z}$-graded
Leibniz algebra given in Example \ref{Lq1}.
Comparing the coefficients of $\varepsilon_{3}t^{n+m}\otimes \varepsilon_{1}t^{n-m}$ and
$\varepsilon_{1}t^{n+m}\otimes \varepsilon_{4}t^{n-m}$ respectively
 in the expansion
 \begin{align*}&(\delta-\hat{\tau}\delta)((a\otimes \varepsilon_{2}t^{n})\circ (b\otimes \varepsilon_{2}t^{n}))
 -(L_{\cdot}(a\otimes \varepsilon_{2}t^{n})\hat{\otimes} I)(\delta
-\hat{\tau}\delta)(b\otimes \varepsilon_{2}t^{n})
-(I\hat{\otimes} L_{\cdot}(a\otimes \varepsilon_{2}t^{n}))(\delta-\hat{\tau}\delta)(b\otimes \varepsilon_{2}t^{n})
\\&-(I\hat{\otimes} R_{\cdot}(b\otimes \varepsilon_{2}t^{n}))\delta(a\otimes \varepsilon_{2}t^{n})
 -( R_{\cdot}(b\otimes \varepsilon_{2}t^{n})\hat{\otimes} I)\hat{\tau}\delta(a\otimes \varepsilon_{2}t^{n})=0
,\end{align*}
which indicates that Eqs.~\eqref{Zd4} and \eqref{Zd6} hold.
Comparing the coefficients of $\varepsilon_{1}t^{m}\otimes \varepsilon_{4}t^{2n-m}$ and
$\varepsilon_{4}t^{m}\otimes \varepsilon_{1}t^{2n-m}$ respectively
 in the expansion
\begin{small} \begin{align*}&
\delta\Big((a\otimes \varepsilon_{2}t^{n})\cdot (b\otimes \varepsilon_{1}t^{n})-(b\otimes \varepsilon_{1}t^{n})\cdot (a\otimes \varepsilon_{2}t^{n})\Big)-\Big((I\otimes (R_{\cdot}-L_{\cdot})(b\otimes \varepsilon_{1}t^{n}))\Big)\delta(a\otimes \varepsilon_{2}t^{n})
\\&-\Big(I \otimes(L_{\cdot}-R_{\cdot})(a\otimes \varepsilon_{2}t^{n})\Big)\delta(b\otimes \varepsilon_{1}t^{n})
-\Big(L_{\cdot}(a\otimes \varepsilon_{2}t^{n})\otimes I\Big)\delta(b\otimes \varepsilon_{1}t^{n})+\Big(L_{\circ}(b\otimes \varepsilon_{1}t^{n})\otimes I\Big)\delta(a\otimes \varepsilon_{2}t^{n})=0,\end{align*}
\end{small}
which indicates that Eqs.~\eqref{Zd5} and \eqref{Zd7} hold.
Comparing the coefficients of $\varepsilon_{3}t^{m}\otimes \varepsilon_{4}t^{2n-m}$ and
$\varepsilon_{4}t^{m}\otimes \varepsilon_{3}t^{2n-m}$ respectively
 in the expansion
\begin{small} \begin{align*}&
\delta\Big((a\otimes \varepsilon_{2}t^{n})\cdot (b\otimes \varepsilon_{3}t^{n})-(b\otimes \varepsilon_{3}t^{n})\cdot (a\otimes \varepsilon_{2}t^{n})\Big)-\Big((I\otimes (R_{\cdot}-L_{\cdot})(b\otimes \varepsilon_{3}t^{n}))\Big)\delta(a\otimes \varepsilon_{2}t^{n})
\\&-\Big(I \otimes(L_{\cdot}-R_{\cdot})(a\otimes \varepsilon_{2}t^{n})\Big)\delta(b\otimes \varepsilon_{3}t^{n})
-\Big(L_{\cdot}(a\otimes \varepsilon_{2}t^{n})\otimes I\Big)\delta(b\otimes \varepsilon_{3}t^{n})+\Big(L_{\circ}(b\otimes \varepsilon_{3}t^{n})\otimes I\Big)\delta(a\otimes \varepsilon_{2}t^{n})=0,\end{align*}
\end{small}
which indicates that Eqs.~\eqref{Zd3} and \eqref{Zd7} hold.
Thus, $(P, \rhd,\lhd,\beta_{\rhd},\beta_{\lhd})$ is a Zinbie-dendriform bialgebra.
This finishes the proof.
 \end{proof}
 
Infinite-dimensional pre-Lie bialgebras can be obtained via the affinization of
Zinbiel-dendriform bialgebras. It is therefore natural to ask whether the solutions
 of the ZD-YBE with invariant skew-symmetric parts
induce solutions of the $S$-equation with invariant symmetric parts in the resulting pre-Lie algebra. 
 This is precisely the question we address in what follows.
  
 \begin{thm} \label{Tb8}
Let $(P,\rhd,\lhd)$ be a Zinbiel-dendriform algebra and $(Q=\oplus_{i\in\mathbb{Z}}Q_{i}, \circ,
\varpi)$ be a quadratic $\mathbb{Z}$-graded Leibniz algebra. Assume that $(P\otimes Q, \cdot)$ is the
induced $\mathbb{Z}$-graded pre-Lie algebra from $(P,\rhd,\lhd)$  
by $(Q=\oplus_{i\in\mathbb{Z}}Q_{i}, \circ,\varpi)$.
 Suppose that $r=\sum_{i}u_{i}
\otimes v_{i}\in P\otimes P$ is a solution of the ZD-YBE in  $(P,\rhd,\lhd)$ with $r+\tau(r)$ invariant in $P\otimes P$. Then
\begin{align}\label{Cr2} 
\hat{r}=\sum_{i}\sum_{j\in\Omega}(u_{i}\otimes e_{j})\otimes(v_{i}\otimes f_{j})
\in(P\otimes Q)\,\hat{\otimes}\,(P\otimes Q)  
\end{align}
is a completed solution of the $S$-equation in $(P\otimes Q, \cdot)$ with $\hat{r}-\hat{\tau} \hat{r}$
 invariant in $(P\otimes Q)\hat{\otimes}(P\otimes Q)$. In particular, if 
 $r=\sum_{i}u_{i}
\otimes v_{i}\in P\otimes P$ is a skew-symmetric solution of the ZD-YBE in $(P,\rhd,\lhd)$, then
$\hat{r}=\sum_{i}\sum_{j\in\Omega}(u_{i}\otimes e_{j})\otimes(v_{i}\otimes f_{j})
\in(P\otimes Q)\,\hat{\otimes}\,(P\otimes Q) $
is a symmetric completed solution of the $S$-equation in $(P\otimes Q, \cdot)$,
where $\{e_{j}\}_{j\in\Omega}$ is a homogeneous basis of $Q=\oplus_{i\in\mathbb{Z}}Q_{i}$
and $\{f_{j}\}_{j\in\Omega}$ is its homogeneous dual basis with respect to $\omega$. 

Furthermore, if $(Q=\oplus_{i\in\mathbb{Z}}Q_{i}, \circ,
\varpi)$ is the quadratic $\mathbb{Z}$-graded Leibniz algebra given in Example~\ref{Lq1}, then 
    \begin{align*}
       \hat{r} =& \sum_{m, n \in \mathbb{Z}} \sum_{i} (u_{i} \otimes \varepsilon_{1} t^{m}) \otimes (v_{i} \otimes \varepsilon_{3} t^{-m} )
       +(u_{i} \otimes \varepsilon_{2} t^{m}) \otimes (v_{i} \otimes \varepsilon_{4} t^{-m} )\\&
       -(u_{i} \otimes \varepsilon_{3} t^{m}) \otimes (v_{i} \otimes \varepsilon_{1} t^{-m} )
       -(u_{i} \otimes \varepsilon_{4} t^{m}) \otimes (v_{i} \otimes \varepsilon_{2} t^{-m} )
   \end{align*}
    is a completed solution of the $S$-equation in $(P\otimes Q, \cdot)$ with $\hat{r}-\hat{\tau} \hat{r}$
 invariant in $(P\otimes Q)\hat{\otimes}(P\otimes Q)$
     if and only if $r+\tau(r)$ is invariant in $P\otimes P$ and 
    $r=\sum_{i}u_{i}
\otimes v_{i}\in P\otimes P$ is
     a solution of the ZD-YBE in $(P,\rhd,\lhd)$. In particular,
$\hat{r}$ is a symmetric completed solution of the $S$-equation in $(P\otimes Q, \cdot)$ if and only if $r=\sum_{i}u_{i}
\otimes v_{i}\in P\otimes P$ is a skew-symmetric solution of the ZD-YBE in $(P,\rhd,\lhd)$.
\end{thm}

\begin{proof}
Using the left non-degeneracy of $\hat{\varpi}$, \cite{HL} gives the following equations,
 \begin{small}
	\begin{align}
&\label{Bd5}\sum_{p,q\in\Omega}e_{p}\otimes f_{p}\circ e_{q} \otimes f_{q}
 =-\sum_{p,q\in\Omega}e_{p}\otimes e_{q}\otimes f_{p}\circ f_{q}
  = \sum_{p,q\in\Omega}e_{p} \circ e_{q} \otimes f_{p}\otimes f_{q}+ e_{q}\circ e_{p} \otimes f_{p}\otimes f_{q},\\
		&\label{Bd6}\sum_{p\in\Omega} e_{q}\circ e_{p} \otimes f_{p}=\sum_{p\in\Omega}f_{p}\otimes e_{q}\circ e_{p}
		=-\sum_{p\in\Omega}e_{p}\otimes e_{q}\circ f_{p} 
		=-\sum_{p\in\Omega} e_{q}\circ f_{p} \otimes e_{p}, 
\\&\label{Bd7}\sum_{p\in\Omega}e_{p}\otimes f_{p}\circ e_{q} =-\sum_{p\in\Omega}f_{p}\otimes e_{p}\circ e_{q}=
		\sum_{p\in\Omega} e_{q}\circ e_{p} \otimes f_{p} + e_{p} \circ e_{q} \otimes f_{p},
\\&\label{Bd8}\sum_{p,q\in\Omega} e_{q}\circ e_{p} \otimes
		f_{p}\otimes f_{q}=-\sum_{p,q\in\Omega}e_{p}\otimes e_{q}\circ f_{p} \otimes f_{q},
\\&\label{Bd9}\sum_{p,q\in\Omega} e_{p}\circ e_{q} \otimes f_{p}\otimes f_{q}=\sum_{p,q\in\Omega}e_{p}\otimes
		e_{q}\otimes f_{q}\circ f_{p},
 \\&\label{Bd10}\sum_{p\in\Omega} e_{p}\circ e_{q} \otimes f_{p}=-\sum_{p\in\Omega} f_{p}\circ e_{q} \otimes e_{p}.
	\end{align} 
\end{small}
By replacing $(A, \Delta_{\succ},\Delta_{\prec})$ and $(B,\eta)$ with
 $(P, \beta_{\rhd},\beta_{\lhd})$ and $(Q,\vartheta)$ respectively
 in the proof of Theorem \ref{Tb4}, and  in combination with Eqs.~\eqref{Bd5}-\eqref{Bd10} and \eqref{Yq1}-\eqref{Yq2}, we have 
\begin{small} \begin{align*}
&(L_{\cdot}(a\otimes e_k)\hat{\otimes} I+I\hat{\otimes}(L_{\cdot}-R_{\cdot})(a\otimes e_k) )(\hat{r}-\hat{\tau}\hat{r})
\\=&\sum_{i}\sum_{j} (a\otimes e_k)\cdot(u_i\otimes e_j)\otimes v_i\otimes f_j-(a\otimes e_k)\cdot(v_i\otimes f_j)\otimes u_i\otimes e_j
\\&+u_i\otimes e_j\otimes[a\otimes e_k,v_i\otimes f_j]-v_i\otimes f_j\otimes[a\otimes e_k,u_i\otimes e_j]
\\=&\sum_{j}(L_{\rhd}(a)\otimes I-I\otimes (L_{\ast}+R_{\ast})(a))(r+\tau(r))\bullet \big(e_k\circ e_j\otimes f_j\big)
\\&-(R_{\lhd}(a)\otimes I+I\otimes R_{\ast}(a))(r+\tau(r))\bullet \big(e_j\circ e_k\otimes f_j\big)\nonumber,
\end{align*}
and
\begin{align*}
S(\hat{r})=\hat{r}_{12}\cdot\hat{r}_{13}-[\hat{r}_{13},\hat{r}_{23}]
-\hat{r}_{12}\cdot\hat{r}_{23}
=\sum_{p,q}-Z(r)\bullet\big (e_{q}\circ e_{p}\otimes f_{p}\otimes f_{q}\big)
+Z_{1}(r)\bullet \big(e_{p}\circ e_{q}\otimes f_{p}\otimes f_{q}\big),
\end{align*}\end{small}
where $Z_{1}(r)=r_{13}\ast r_{23}+r_{23}\ast r_{13}+r_{12}\rhd r_{13}-r_{12}\rhd r_{23}$.
In view of $r+\tau(r)$ being invariant in $P\otimes P$ and Proposition \ref{Qy}, we get that
 \begin{align*}
&(L_{\cdot}(a\otimes e_k)\hat{\otimes} I+I\hat{\otimes}(L_{\cdot}-R_{\cdot})(a\otimes e_k) )(\hat{r}-\hat{\tau}\hat{r})
=0,\\&S(\hat{r})=\hat{r}_{12}\cdot\hat{r}_{13}-[\hat{r}_{13},\hat{r}_{23}]
-\hat{r}_{12}\cdot\hat{r}_{23}=0.
\end{align*}
Thus,
 $\hat{r}-\hat{\tau}\hat{r}$
 is invariant in $(P\otimes Q)\hat{\otimes}(P\otimes Q)$ and 
 $\hat{r} $ is a completed solution of the $S$-equation $S(\hat{r})=0 $ in $(P\otimes Q, \cdot)$.
It is clear that if $r$ is skew-symmetric, then $\hat{r}$
is symmetric. 

Let $(Q=\oplus_{i\in\mathbb{Z}}Q_{i},\circ,\varpi)$ be the
quadratic $\mathbb{Z}$-graded Leibniz algebra given in Example ~\ref{Lq1}, then 
$\{\varepsilon_3t^{-m},\varepsilon_4t^{-m},\varepsilon_1t^{-m},\varepsilon_2t^{-m}
 \mid m\in \mathbb{Z}\}$ 
 is the dual basis of $\{\varepsilon_1t^{m},\varepsilon_2t^{m},-\varepsilon_3t^{m},-\varepsilon_4t^{m}
 \mid m\in \mathbb{Z}\}$ with respect to
 $\varpi$.  The "if" part is obvious. On the other hand,
assume that $\hat{r}-\hat{\tau}\hat{r}$
 is invariant in $(P\otimes Q)\hat{\otimes}(P\otimes Q)$ and $\hat{r}$ is a completed solution of the S-equation in $(P\otimes Q, \cdot)$.
  Comparing the coefficients of $\varepsilon_1t^{n+m} \otimes \varepsilon_3t^{-m}$ and
   $\varepsilon_1t^{n+m}  \otimes \varepsilon_4t^{-m}$ respectively in the expansion
 \begin{small}\begin{align*}
 (L_{\cdot}(a\otimes \varepsilon_2t^{n})\hat{\otimes} I+I\hat{\otimes}(L_{\cdot}-R_{\cdot})(a\otimes \varepsilon_2t^{n}) )(\hat{r}-\hat{\tau}\hat{r})
 =0,\end{align*}\end{small}
we obtain that 
\begin{small}\begin{align*}&(L_{\rhd}(a)\otimes I-I\otimes (L_{\ast}+R_{\ast})(a))(r+\tau(r))=0,\\&
(R_{\lhd}(a)\otimes I+I\otimes R_{\ast}(a))(r+\tau(r))=0.\end{align*}\end{small}
It follows that
\begin{align*}(I\otimes L_{\ominus}(a)- L_{\ast}(a)\otimes I)(r+\tau(r))=0.\end{align*}
Thus, $r+\tau(r)$ is invariant in $P\otimes P$.

Comparing the coefficients of $\varepsilon_1t^{n+m} \otimes \varepsilon_4t^{-n}\otimes \varepsilon_3t^{-m}$ in the expansion
\begin{align*}
&S(\hat{r})=\hat{r}_{12}\cdot\hat{r}_{13}-[\hat{r}_{13},\hat{r}_{23}]
-\hat{r}_{12}\cdot\hat{r}_{23}=0,\end{align*}
which implies that $Z(r)=0$.
 \end{proof}
 
\begin{thm}
Let $(P, \rhd,\lhd,\beta_{\rhd},\beta_{\lhd})$ be a 
Zinbiel-dendriform bialgebra and $(Q=\oplus_{i\in\mathbb{Z}}Q_{i}, \circ,
\varpi)$ be a quadratic $\mathbb{Z}$-graded Leibniz algebra. Assume that $(P\otimes Q, \cdot,\delta)$ is
the induced completed pre-Lie bialgebra from $(P, \rhd,\lhd,\beta_{\rhd},\beta_{\lhd})$ by
$(Q, \circ,\varpi)$. If $(P, \rhd,\lhd,\beta_{\rhd,r},\beta_{\lhd,r})$ is quasi-triangular,
 then $(P\otimes Q, \cdot,\delta_{\hat{r}})$ is also quasi-triangular.
 In particular, $(P\otimes Q, \cdot,\delta_{\hat{r}})$ is also triangular if 
$(P, \rhd,\lhd,\beta_{\rhd,r},\beta_{\lhd,r})$ is triangular.
Moreover, if $(Q=\oplus_{i\in\mathbb{Z}}Q_{i}, \circ,\varpi)$ is 
the quadratic $\mathbb{Z}$-graded Leibniz algebra given in Example~\ref{Lq1},
    then $(P, \rhd,\lhd,\beta_{\rhd,r},\beta_{\lhd,r})$ is a quasi-triangular Zinbiel-dendriform bialgebra
    if and only if
 $(P\otimes Q, \cdot, \delta_{\hat{r}})$ is a quasi-triangular pre-Lie bialgebra,
 where $r$ and $\hat{r}$ are as given in Theorem \ref{Tb8}.

\end{thm}

\begin{proof} Let $(P, \rhd,\lhd,\beta_{\rhd},\beta_{\lhd})$ be a
quasi-triangular Zinbiel-dendriform bialgebra. Assume that
 $r=\sum_{i}u_{i}\otimes v_{i}$ is a solution of the ZD-YBE $Z(r)=0$ in $(P, \rhd,\lhd)$ 
and $r+\tau(r)$ is invariant in $P\otimes P$.
By Theorem \ref{Tb8}, we know that $\hat{r}$ is a completed solution
 of the $S$-equation $S(\hat{r})=0$ in $(P\otimes Q, \cdot)$ and $\hat{r}-\hat{\tau}\hat{r}$
 is invariant in $(P\otimes Q)\hat{\otimes} (P\otimes Q)$.
It follows that $(P\otimes Q, \cdot, \delta_{\hat{r}})$ is a quasi-triangular
completed pre-Lie bialgebra, where
$\delta_{\hat{r}}$ is given by Eq.~\eqref{CB4}. Moreover, let $(Q=\oplus_{i\in\mathbb{Z}}Q_{i}, \circ,\varpi)$ be 
the quadratic $\mathbb{Z}$-graded Leibniz algebra given in Example~\ref{Lq1}. According to Theorem \ref{Tb8},
$\hat{r}$ is a completed solution of the $S$-equation in $(P\otimes Q, \cdot)$ with $\hat{r}-\hat{\tau} \hat{r}$
 invariant in $(P\otimes Q)\hat{\otimes}(P\otimes Q)$
     if and only if $r+\tau(r)$ is invariant in $P\otimes P$ and $r$ is
a solution of the ZD-YBE in $(P,\rhd,\lhd)$.
Consequently, $(P, \rhd,\lhd,\beta_{\rhd,r},\beta_{\lhd,r})$ is a quasi-triangular Zinbiel-dendriform bialgebra
    if and only if
 $(P\otimes Q, \cdot,\delta_{\hat{r}})$ is a quasi-triangular pre-Lie bialgebra.
 
The proof is finished.
\end{proof}
 
\begin{ex} 
Let $(P,\rhd,\lhd)$ be the $3$-dimensional Zinbiel-dendriform algebra with a basis $\{u,v,w\}$ as given in Example \ref{Lq3},
and let
$(Q,\circ,\varpi)$ be the quadratic $\mathbb Z$-graded Leibniz algebra given in Example \ref{Lq1}. 
Then by Theorem \ref{Tb8}, 
\begin{align*}
       \hat{r} =& \sum_{m, n \in \mathbb{Z}}  (v \otimes \varepsilon_{1} t^{m}) \otimes (w \otimes \varepsilon_{3} t^{-m} )
       +(v \otimes \varepsilon_{2} t^{m}) \otimes (w \otimes \varepsilon_{4} t^{-m} )\\&
       -(v \otimes \varepsilon_{3} t^{m}) \otimes (w \otimes \varepsilon_{1} t^{-m} )
       -(v \otimes \varepsilon_{4} t^{m}) \otimes (w \otimes \varepsilon_{2} t^{-m} )\\&
       -(w \otimes \varepsilon_{1} t^{m}) \otimes (v \otimes \varepsilon_{3} t^{-m} )
       -(w \otimes \varepsilon_{2} t^{m}) \otimes (v \otimes \varepsilon_{4} t^{-m} )\\&
       +(w \otimes \varepsilon_{3} t^{m}) \otimes (v \otimes \varepsilon_{1} t^{-m} )
       +(w \otimes \varepsilon_{4} t^{m}) \otimes (v \otimes \varepsilon_{2} t^{-m} )
   \end{align*}
    is a completed symmetric solution of the $S$-equation in $(P\otimes Q, \cdot)$.
\end{ex}
 
The following result arises naturally.

 \begin{thm}\label{Ty2}
 Under the same assumption as in Theorem \ref{Tb8}, we define linear maps
 $\beta_{\rhd,r},\beta_{\lhd,r}:P\rightarrow P\otimes P$ by Eqs.~\eqref{CB4}-\eqref{CB5}, and
$\delta: P\otimes Q \rightarrow (P\otimes Q ) \hat{\otimes} (P\otimes Q )$ 
 by Eq.~\eqref{P2}. Then, by Theorem \ref{Tb6}, the triple $(P\otimes Q, \cdot, \delta)$
 forms a pre-Lie bialgebra. Moreover, this is precisely the quasi-triangular pre-Lie bialgebra
 $(P\otimes Q, \cdot, \delta_{\hat{r}})$ with cobracket defined by Eq.~\eqref{CB3}, 
 where $\hat{r}$ is given by Eq.~\eqref{Cr2}. In other words, we have the following commutative diagram:
 \begin{equation*}
        \xymatrix@C=3cm{
            \txt{ solutions of the ZD-YBE \\ with invariant skew-symmetric parts} \ar[r]^-{Theorem~\ref{Qy1}}
             \ar[d]^-{Theorem ~\ref{Tb7}} & \txt{Zinbiel-dendriform bialgebras} \ar[d]^-{Theorem~\ref{Tb8}}\\
            \txt{solutions of the S-equation \\ with invariant symmetric parts} 
            \ar[r]^-{Theorem~\ref{Py}} & \txt{pre-Lie bialgebras}}
    \end{equation*}

\end{thm}
\begin{proof}
The same argument as in Theorem \ref{Ty1} applies here.

\end{proof}

\begin{center}{\textbf{Acknowledgments}}
\end{center}
Q. Sun is supported by the Natural Science
Foundation of Zhejiang Province of China (No. LY19A010001) and the Science
and Technology Planning Project of Zhejiang Province
(No. 2023C01130).

\begin{center} {\textbf{Statements and Declarations}}
\end{center}
 All datasets underlying the conclusions of the paper are available
to readers. No conflict of interest exits in the submission of this
manuscript.


\end {document}